\documentclass[12pt,a4paper,reqno]{amsart}
\usepackage[english]{babel}
\usepackage[applemac]{inputenc}
\usepackage[T1]{fontenc}
\usepackage{palatino}
\usepackage{amsmath}
\usepackage{amssymb}
\usepackage{amsthm}
\usepackage{amsfonts}
\usepackage{graphicx}
\usepackage[colorlinks = true, citecolor = black]{hyperref}
\author{Tuomas Orponen}\thanks{The research was supported by the Finnish Centre of Excellence in Analysis and Dynamics Research}
\title[Packing Dimension and Category of Exceptions]{On the Packing Dimension and Category of Exceptional Sets of Orthogonal Projections}
\address{Department of Mathematics and Statistics, University of Helsinki, P.O.B. 68, FI-00014 University of Helsinki, Finland}
\email{tuomas.orponen@helsinki.fi} 

\newcommand{\R}{\mathbb{R}}
\newcommand{\N}{\mathbb{N}}

\newcommand{\Z}{\mathbb{Z}}

\newcommand{\calP}{\mathcal{P}}

\newcommand{\calT}{\mathcal{T}}
\newcommand{\calI}{\mathcal{I}}
\newcommand{\calH}{\mathcal{H}}

\newcommand{\calG}{\mathcal{G}}
\newcommand{\calL}{\mathcal{L}}
\newcommand{\calQ}{\mathcal{Q}}

\newcommand{\spt}{\operatorname{spt}}
\newcommand{\card}{\operatorname{card}}

\newcommand{\p}{\mathbf{p}}
\newcommand{\B}{\operatorname{B}}
\newcommand{\MB}{\operatorname{MB}}
\newcommand{\Dim}{\operatorname{Dim}}
\newcommand{\prev}{\operatorname{prev}}
\newcommand{\m}{\mathfrak{m}}

\numberwithin{equation}{section}

\theoremstyle{plain}
\newtheorem{thm}[equation]{Theorem}
\newtheorem{lemma}[equation]{Lemma}

\newtheorem{cor}[equation]{Corollary}
\newtheorem{proposition}[equation]{Proposition}
\newtheorem{question}[equation]{Question}
\newtheorem{claim}[equation]{Claim}

\theoremstyle{definition}
\newtheorem{definition}[equation]{Definition}

\newtheorem{construction}[equation]{Construction}

\theoremstyle{remark}
\newtheorem{remark}[equation]{Remark}

\newtheorem{notation}[equation]{Notations}

\begin{document}

\begin{abstract} We consider several classical results related to the Hausdorff dimension of \emph{exceptional sets} of orthogonal projections and try to find out whether they have reasonable formulations in terms of packing dimension. We also investigate the existence of category versions for Marstrand and Falconer-Howroyd-type projection results.
\end{abstract}

\maketitle

\tableofcontents

\section{Introduction} Given a set $K \subset \R^{2}$, what is the relation between the Hausdorff or packing dimension of $K$, and the Hausdorff or packing dimension of the generic orthogonal projection $K_{e} = \{x \cdot e : x \in K\}$, for $e \in S^{1}$? This is one of the most classical and thoroughly studied questions in geometric measure theory. As early as 1954, J.M. Marstrand \cite{Mar} proved that Hausdorff dimension is generally preserved in projections. More precisely, if the Hausdorff dimension of $K$, denoted by $\dim K$, is at most one, then $\dim K_{e} = \dim K$ for almost every vector $e \in S^{1}$. In case $\dim K > 1$, the result fails for obvious reasons, but, instead, Marstrand proved that almost every projection has positive length. 

Things change radically when hypotheses on Hausdorff dimension are replaced by those on packing dimension, denoted by $\dim_{\p}$. A special case of a construction due to M. J\"arvenp\"a\"a \cite{Jar} yields for any $\gamma \in (0,2)$ a compact set $K$ in the plane such that $\dim_{\p} K = \gamma$, yet $\dim_{\p} K_{e} \leq 2\gamma/(2 + \gamma) < \gamma$ for every $e \in S^{1}$. A few years later, it was discovered by K. Falconer and J. Howroyd in \cite{FH} that the behavior seen here is essentially the worst possible: any analytic set $K \subset \R^{2}$ with $\dim_{\p} K = \gamma$ has $\dim_{\p} K_{e} \geq 2\gamma/(2 + \gamma)$ for almost every $e \in S^{1}$. Moreover, the function $e \mapsto \dim_{\p} K_{e}$ is almost surely constant.

So, there is nothing new about studying the dimensions of orthogonal projections. Neither is it news that the results of Marstrand and Falconer-Howroyd can be sharpened by examining the dimension of \emph{exceptional sets of projections}. Given $K \subset \R^{2}$, such a set is formed by the directions $e \in S^{1}$ where the 'expected' behavior of $\dim K_{e}$ or $\dim_{\p} K_{e}$ fails. An early result on the dimension of exceptional sets is a theorem of Kaufman \cite{Ka} from 1969 saying that if $K \subset \R^{2}$ is an analytic set, then
\begin{equation}\label{kaufman} \dim \{e \in S^{1} : \dim K_{e} \leq \sigma\} \leq \sigma, \qquad 0 \leq \sigma < \dim K.  \end{equation} 
In particular, it follows that $\dim \{e \in S^{1} : \dim K_{e} < \dim K\} \leq \dim K$, which is sharp according to an example of R. Kaufman and P. Mattila \cite{KM}. In a similar vein but with a completely different technique, J. Bourgain \cite[Theorem 4]{Bo} proved in 2010 that if a set $K \subset \R^{2}$ has Hausdorff dimension $\dim K > \alpha \in (0,2)$, then
\begin{equation}\label{bourgain} \dim \{e \in S^{1} : \dim K_{e} < \eta\} \leq \kappa(\alpha,\eta), \end{equation}
where $\kappa(\alpha,\eta) \to 0$ as $\eta \searrow \alpha/2$. In case $\dim K > 1$, we mention the sharp bound
\begin{displaymath} \dim \{e \in S^{1} : \calH^{1}(K_{e}) = 0\} \leq 2 - \dim K, \end{displaymath}
due to K. Falconer \cite{Fa}. The estimates of Kaufman and Falconer were generalized to a much richer class of 'projections' than merely orthogonal ones in an influential paper of Y. Peres and W. Schlag \cite{PSc} in 2000. In \cite{FH}, Falconer and Howroyd improve on their own 'almost all' results by estimating the Hausdorff dimension of the exceptional sets related to the conservation of packing dimension under orthogonal projections. The sharp bounds are unknown in this situation but, for example, their results imply that
\begin{equation}\label{falc} \dim \left\{e \in S^{1} : \dim_{\p} K_{e} < \frac{\dim_{\p} K}{1 + (1/\sigma - 1/2)\dim_{\p} K} \right\} \leq \sigma, \qquad 0 \leq \sigma \leq 1. \end{equation} 

All estimates cited above are formulated in terms of the Hausdorff dimension of the exceptional sets under consideration. The starting point of this paper is to investigate if similar bounds could be obtained in terms of packing dimension. Since $\dim B \leq \dim_{\p} B$ for any set $B \subset \R^{d}$, bounds for $\dim_{\p}$ can certainly be no lower than those for $\dim$. But, to begin with, it is not even clear if one can hope for \textbf{any} non-trivial estimates for the packing dimension of exceptional sets. The only existing result in any direction seems to be due to M. Rams \cite{Ra} from 2002. It is concerned with the dimensions of self-conformal fractals $\Lambda_{t} \subset \R^{d}$, which vary smoothly and \emph{transversally} (see \cite[Theorem 1.1]{Ra} and the references therein for the definitions) as the parameter $t$ takes values in some open subset $V \subset \R^{d}$. Rams proves that for every $u \in V$ there exists a number $s(u) \geq 0$ (defined in terms of Bowen's equation, equal to the similarity dimension of $\Lambda_{u}$ in case the conformal mappings are similitudes) such that
\begin{displaymath} \limsup_{r \to 0} \dim_{\p} \{t \in B(u,r) : \dim \Lambda_{t} \leq \sigma\} \leq \sigma, \qquad \sigma < \min\{d,s(u)\}. \end{displaymath}
In order to better connect Rams' result to orthogonal projections, let us formulate a special case, which follows immediately from the inequality above. If $K \subset \R^{2}$ is a self-similar set in the plane satisfying the strong separation condition and containing no rotations, then
\begin{equation}\label{rams} \dim_{\p} \{e \in S^{1} : \dim K_{e} \leq \sigma\} \leq \sigma, \qquad \sigma < \dim K. \end{equation} 
This is precisely Kaufman's bound \eqref{kaufman} with one $\dim$ replaced by $\dim_{\p}$! The content of our first result is that such an improvement for \eqref{kaufman} is not possible for general sets.
\begin{thm}\label{main} There exists a compact set $K \subset \R^{2}$ with $\calH^{1}(K) > 0$ such that $\dim K_{e} = 0$ in a dense $G_{\delta}$-set of directions $e$. 
\end{thm}

Dense $G_{\delta}$-sets on $S^{1}$ always have packing dimension one, so this shows that (a) the exceptional set estimate \eqref{rams} cannot be stated for general sets, and (b) the bounds \eqref{bourgain} of Bourgain and \eqref{kaufman} of Kaufman cannot be formulated in terms of packing dimension. Next, we ask what happens if $\dim K_{e}$ is replaced by $\dim_{\p} K_{e}$, that is, can we obtain bounds for $\dim_{\p} \{e \in S^{1} : \dim_{\p} K_{e} \leq \sigma\}$? An example as dramatic as the one in Theorem \ref{main} is not possible now because of
\begin{proposition}\label{mainP} Let $K \subset \R^{2}$ be an analytic set with $\dim_{\p} K = s$, and let $e,\xi \in S^{1}$ be two linearly independent vectors. Then
\begin{displaymath} s \leq \dim_{\p} K_{e} + \dim_{\p} K_{\xi}. \end{displaymath}
In particular, 
\begin{displaymath} \card \left\{e \in S^{1} : \dim_{\p} K_{e} < \tfrac{s}{2}\right\} \leq 2. \end{displaymath}
\end{proposition}
This proposition is a special case of a result in \cite{Jar}; one may view it as a generalization of the well-known inequality $\dim_{\p}(A \times B) \leq \dim_{\p} A + \dim_{\p} B$ for the packing dimension of product sets, see \cite[Theorem 8.10(3)]{Ma}. In light of Proposition \ref{mainP}, the worst behavior imaginable is this: a set $K \subset \R^{2}$ with packing dimension $\dim_{\p} K = \gamma$ is projected to a set of packing dimension $\gamma/2$ in a set $E \subset S^{1}$ containing (many) more than two directions. On the other hand, it follows from the bound \eqref{falc} that $\dim E \leq 2\gamma/(2 + \gamma)$, so $E$ cannot be very large in terms of Hausdorff dimension. Our next result shows that $E$ can have full packing dimension:
\begin{thm}\label{main2} Given $\gamma \in [0,2]$, there exists a compact set $K \subset \R^{2}$ with $\dim_{\p} K = \gamma$ such that
\begin{equation}\label{form1} \dim_{\p} \left\{e \in S^{1} : \dim_{\p} K_{e} = \frac{\gamma}{2}\right\} = 1. \end{equation} 
\end{thm}
This answers -- in the plane -- a question on the packing dimension of exceptional sets raised in \cite[\S4]{FH1}. In contrast with the example in Theorem \ref{main}, we cannot hope to construct compact -- or even analytic -- sets $K \subset \R^{2}$ such that $0 < \dim_{\p} K < 2$, and $\{e \in S^{1} : \dim_{\p} K_{e} = \dim_{\p} K/2\}$ has second category.
\begin{thm}\label{packingCategory} Let $A \subset \R^{2}$ be an analytic set, and let
\begin{displaymath} \m := \sup \{\dim_{\p} A_{e} : e \in S^{1}\}. \end{displaymath}
Then $\{e \in S^{1} : \dim_{\p} A_{e} \neq \m \}$ is a meagre set with with zero length.
\end{thm}
\begin{remark}\label{FHRemark} The 'zero length' part of the theorem follows from \cite{FH}. Namely, if $\dim_{\p} A = \gamma \in [0,2]$, it was shown in \cite{FH} that there exists a constant $c \geq 2\gamma/(2 + \gamma)$ such that $\dim_{\p} A_{e} \leq c$ for all $e \in S^{1}$, and $\dim_{\p} A_{e} = c$ for almost all $e \in S^{1}$. Of course, this implies that $c = \m$. In case $0 < \gamma < 2$, we then have $\m = c \geq 2\gamma/(2 + \gamma) > \gamma/2$, and, in particular, the set $\{e : \dim_{\p} A_{e} = \dim_{\p} A/2\}$ is meagre for $0 < \dim_{\p} A < 2$. Our proof of Theorem \ref{packingCategory} -- very different from the one in \cite{FH} -- gives the same result for upper box dimension as a by-product, see Theorem \ref{boxCategory}. This was not contained in \cite{FH}, but (the zero-length part of) the result was proved by Howroyd \cite{Ho} in 2001, developing further the potential theoretic machinery from \cite{FH}. 
\end{remark}
In view of Theorems \ref{main} and \ref{main2}, it might seem that packing dimension is a hopelessly inaccurate tool for measuring the size of exceptional sets. However, there is one more direction unexplored. If the set $K \subset \R^{2}$ has large Hausdorff dimension to begin with, what can we say about the set $\{e \in S^{1} : \dim_{\p} K_{e} \leq \sigma\}$? In this situation, the only existing general result seems to be the following one by Peres, K. Simon and B. Solomyak \cite[Proposition 1.3]{PSS}. If $K \subset \R^{2}$ is an analytic set with $\calH^{\gamma}(K) > 0$ for some $\gamma \in (0,1]$, then
\begin{equation}\label{PSSEstimate} \dim \{e \in S^{1} : \mathcal{P}^{\gamma}(K_{e}) = 0\} \leq \gamma. \end{equation}
In Peres, Simon and Solomyak's result, the size of the exceptional set is still measured in terms of Hausdorff dimension. Our fourth theorem provides an estimate for the \textbf{packing} dimension of the exceptional set $\{e \in S^{1} : \dim_{\p} K_{e} \leq \sigma\}$: 
\begin{thm}\label{estimate} Let $K \subset \R^{2}$ be an analytic set with Hausdorff dimension $\dim K = \gamma \in (0,1]$. Then we have the estimates
\begin{displaymath} \dim_{\p}\{e \in S^{1} : \dim_{\p} K_{e} \leq \sigma\} \leq \frac{\sigma\gamma}{\gamma + \sigma(\gamma - 1)}, \qquad 0 \leq \sigma \leq \gamma, \end{displaymath}
and
\begin{displaymath} \dim_{\p}\{e \in S^{1} : \dim_{\p} K_{e} \leq \sigma\} \leq \frac{(2\sigma - \gamma)(1 - \gamma)}{\gamma/2} + \sigma, \qquad \gamma/2 \leq \sigma \leq \gamma. \end{displaymath}  
\end{thm}
\begin{remark} The bounds may be difficult to read at first sight, so let us review some of their features. First, the restriction $\sigma \geq \gamma/2$ in the second bound has little consequence, since, by Proposition \ref{mainP}, the exceptional set $\{e \in S^{1} : \dim_{\p} K_{e} \leq \sigma\}$ has anyway packing dimension zero for $0 \leq \sigma < \gamma/2$. The first estimate is sharper than the second for $\sigma$ close to $\gamma$: the upper bound in the first estimate is less than $\sigma/\gamma < 1$ for $\sigma < \gamma$, and tends to one as $\sigma \nearrow \gamma$; the second estimate unfortunately tends to $2 - \gamma \geq 1$. Naturally, the second estimate outperforms the first one for $\sigma$ close to $\gamma/2$: the first estimate tends to $\gamma/(1 + \gamma)$ as $\sigma \searrow \gamma/2$, whereas the second estimate recovers the bound $\dim_{\p} \{e : \dim_{\p} K_{e} \leq \gamma/2\} \leq \gamma/2$, which, for self-similar sets, is precisely \eqref{rams}. Finally, the first estimate can be reformulated as follows: if $\tau \leq 1$, then
\begin{displaymath} \dim_{\p} \{e \in S^{1} : \dim_{\p} K_{e} \leq \tau \dim K\} \leq \frac{\tau \cdot \dim K}{\tau \cdot \dim K + (1 - \tau)}. \end{displaymath}
In particular, the bound tends to zero as $\dim K \to 0$.
\end{remark}

The first estimate in Theorem \ref{estimate} shows that $\dim_{\p} \{e \in S^{1} : \dim_{\p} K_{e} \leq \sigma\} < 1$ for any $\sigma < \dim K$, given that $0 \leq \dim K \leq 1$. Since sets with packing dimension less than one are meager, we obtain
\begin{cor} If $0 \leq \dim K \leq 1$, the set $\{e \in S^{1} : \dim_{\p} K_{e} < \dim K\}$ is meager.
\end{cor}

Finally, our method for general sets combined with a 'dimension conservation principle' due to H. Furstenberg \cite{Fu} from 2008 can be used to recover a different proof for -- and a slightly generalized version of -- Rams' estimate \eqref{rams}.

\begin{thm}\label{estimate2} Let $K$ be a self-similar or a compact homogeneous set (see the remark below) in the plane. Then
\begin{displaymath} \dim_{\p} \{e \in S^{1} : \dim K_{e} \leq \sigma\} \leq \sigma, \qquad 0 \leq \sigma < \dim K. \end{displaymath}
\end{thm}

\begin{remark} In contrast with the formulation of Rams' estimate \eqref{rams}, we impose no conditions on separation or the absence of rotations in case the set $K \subset \R^{2}$ above is self-similar. Still, Rams' estimate is -- in the self-similar case -- not essentially less general than the one above: our proof of Theorem \ref{estimate2} starts by reducing the situation to the 'no rotations, strong separation' case. However, one needs results more recent than Rams' paper to accomplish this reduction; namely, we use \cite[Theorem 5]{PS} by Peres and P. Shmerkin from 2009, showing that any orthogonal projection of a planar self-similar set containing an irrational rotation preserves dimension. The \emph{homogeneous sets} mentioned in the statement of Theorem \ref{estimate2} were introduced by H. Furstenberg. Self-similar sets satisfying the strong separation condition and containing no rotations are (not the only) examples of such sets, see \cite[\S1]{Fu}.
\end{remark}

It appears to be a challenging task to figure out the sharpness of Theorems \ref{estimate} and \ref{estimate2}. Here is the best construction we could come up with:
\begin{thm}\label{bigEx} Let $\sigma \in (3/4,1)$. Then there exists a compact set $K \subset \R^{2}$ with $\calH^{1}(K) > 0$, and a number $\tau(\sigma) < 1$ such that
\begin{displaymath} \dim_{\p} \{e \in S^{1} : \dim_{\p} K_{e} \leq \tau(\sigma)\} \geq \sigma. \end{displaymath}
\end{thm}

Thus, one cannot expect very dramatic improvements to Theorem \ref{estimate} -- such as $\dim_{\p} \{e \in S^{1} : \dim_{\p} K_{e} < \dim K\} = 0$ -- but we still strongly suspect that our bounds are not sharp: at any rate, we believe that the packing dimension of the exceptional set $\{e \in S^{1} : \dim_{\p} K_{e} \leq \sigma\}$ should tend to zero as $\sigma \searrow \dim K/2$, in analogue with Bourgain's bound \eqref{bourgain} for Hausdorff dimension. During our futile attempts to verify this conjecture, we came up with the following Marstrand-type theorem for finite planar sets. We had not encountered the result previously, so we state it here and provide a quick proof (based on the Szemer\'edi-Trotter incidence bound) at the end of the paper:
\begin{proposition}\label{pointMarstrand} Let $P \subset \R^{2}$ be a collection of $n \geq 2$ points, and let $1/2 \leq s < 1$. Then
\begin{displaymath} \card \{e \in S^{1} : \card P_{e} \leq n^{s} \} \lesssim_{s} n^{2s - 1}. \end{displaymath}
\end{proposition}


\section{Notations, definitions and the proof of Theorem \ref{main}}

\begin{notation} The unit circle $\{x \in \R^{2} : |x| = 1\}$ is denoted by $S^{1}$. The orthogonal projection in $\R^{2}$ onto the vector spanned by $e \in S^{1}$ is denoted by $\rho_{e}$. For convenience, we think of $\rho_{e}$ as a mapping onto $\R$ instead of $\operatorname{span}(e) \subset \R^{2}$, which means that we define $\rho_{e}(x) := x \cdot e \in \R$ for $x \in \R^{2}$. In agreement with the notation we adopted in the introduction, we will often use the abbreviation $K_{e} := \rho_{e}(K)$ for sets $K \subset \R^{2}$. If $A,B > 0$, the notation $A \lesssim B$ means that $A \leq CB$ for some constant $C \geq 1$, which may depend on various parameters, but not on $B$. 
\end{notation}

Next, we recall some basic facts on packing and box-counting dimensions.
\begin{definition}[Packing and box-counting dimensions]\label{dimensions} If $B \subset \R^{d}$ is any bounded set and $\delta > 0$, we denote by $P(B,\delta)$ the maximum cardinality of a \emph{$\delta$-packing of $B$ with balls}, that is, 
\begin{displaymath} P(B,\delta) := \max \{j \geq 1 : x_{1},\ldots,x_{j} \in B, \text{ and the balls } B(x_{j},\delta) \text{ are disjoint}\}. \end{displaymath}
Under the same setting, we denote by $N(B,\delta)$ the minimum cardinality of a \emph{$\delta$-cover of $B$ with balls}, that is,
\begin{displaymath} N(B,\delta) := \min \left\{j \geq 1 : x_{1},\ldots,x_{j} \in \R^{d}, \text{ and } B \subset \bigcup_{i = 1}^{j} B(x_{i},\delta) \right\}. \end{displaymath}
Since $N(B,2\delta) \leq P(B,\delta) \leq N(B,\delta/2)$, the numbers
\begin{displaymath} \limsup_{\delta \to 0} \frac{\log N(B,\delta)}{-\log \delta} \quad \text{and} \quad \limsup_{\delta \to 0} \frac{\log P(B,\delta)}{-\log \delta} \end{displaymath}
are equal, and the common value is the (upper) \emph{box-counting dimension of $B$}, denoted by $\overline{\dim}_{\B} B$. The packing dimension of $B$ is now defined by
\begin{displaymath} \dim_{\p} B := \inf \left\{\sup_{j} \overline{\dim}_{\B} F_{j} : B \subset \bigcup_{j \in \N} F_{j} \right\}. \end{displaymath}
Since $\overline{\dim}_{\B} B = \overline{\dim}_{\B} \overline{B}$ for any set $B$, the definition above is unaffected, if we assume that the sets $F_{j}$ are closed. 
\end{definition}

It is immediate from the definition of packing dimension that $\dim_{\p} B \leq \overline{\dim}_{\B} B$. The converse inequality is not true in general, but the following proposition from \cite{Fa2} often solves the issue:
\begin{proposition}[Proposition 3.6 in \cite{Fa2}]\label{falcProp} Assume that $K \subset \R^{d}$ is compact, and
\begin{displaymath} \overline{\dim}_{\B}(K \cap U) = \overline{\dim}_{\B} K \end{displaymath}
for all open sets $U$ that intersect $K$. Then $\dim_{\p} K = \overline{\dim}_{\B} K$.
\end{proposition}

In association with Theorem \ref{main}, we claimed that dense $G_{\delta}$-sets on the circle always have packing dimension one. In fact, the same is true for any set $B \subset S^{1}$ of the second category. To see this, cover $B$ with a countable collection of sets $F_{j}$. By definition of second category, $B$ cannot be expressed as the countable union of nowhere dense sets. This implies that the closure of $B \cap F_{j}$ must have non-empty interior for some $j$. In particular, $\overline{\dim}_{\B} F_{j} = 1$, which gives $\dim_{\p} B = 1$. 
 

\begin{proof}[Proof of Theorem \ref{main}] Choose a countable dense set of directions $\{e_{1},e_{2},\ldots,\} \subset S^{1}$, and choose a sequence $(s_{j})_{j \in \N}$ such that $s_{j} \searrow 0$ as $j \to \infty$. Here is the plan. To every vector $e_{m}$, we will eventually associate countably many open arcs $J(e_{m},n)$, $n \geq 1$. The dense $G_{\delta}$-set $G \subset S^{1}$ will be defined by $G = \bigcap U_{n}$, where
\begin{displaymath} U_{n} := \bigcup_{m = 1}^{\infty} J(e_{m},n). \end{displaymath}
The set $K$ will be constructed so that
\begin{equation}\label{form28} \calH^{s_{n}}_{1/n}(K_{e}) \leq 1, \qquad e \in J(e_{m},n), \: m,n \in \N. \end{equation}
This will evidently force $\dim K_{e} = 0$ for every direction $e \in G$. We order the pairs $(e_{m},n)$ according to the following scheme:
\begin{align} & (e_{1},1)\notag\\
& (e_{1},2) \quad (e_{2},1)\notag\\
&\label{form27} (e_{1},3) \quad (e_{2},2) \quad (e_{3},1)\\
& (e_{1},4) \quad (e_{2},3) \quad (e_{3},2) \quad (e_{4},1)\notag\\
& \ldots. \notag\end{align} 
We start moving through the pairs $(e_{m},n)$ in the order indicated by \eqref{form27} -- that is, top down and from left to right. Whenever we encounter a pair $(e_{m},n)$, we will associate to it (i) an arc $J(e_{m},n)$ containing $e_{m}$, and (ii) a compact set $K(e_{m},n)$, which is the finite union of closed balls with a common diameter and disjoint interiors. The sets $K(e_{m},n)$ will all be nested, and so
\begin{displaymath} K := \bigcap_{m,n} K(e_{m},n) \end{displaymath}
will be a compact subset of $\R^{2}$.

To get the recursive procedure started, we define $K(e_{1},1) := B(0,1/2)$ and $J(e_{1},1) = S^{1}$. Then \eqref{form28} is satisfied for $m = n = 1$, no matter what $s_{1}$ is. Then, assume that we have just finished constructing the set $K_{\prev} := K(e_{m(\prev)},n(\prev))$ for some $m(\prev),n(\prev) \in \N$. We assume that $K_{\prev}$ is the union of $p \in \N$ closed balls with disjoint interiors and a common diameter $d > 0$. Let $(e_{m},n)$ be the 'next' pair in the ordering \eqref{form27}. Thus,
\begin{displaymath} (e_{m},n) = (e_{m(\prev) + 1},n(\prev) - 1) \quad \text{or} \quad (e_{m},n) = (e_{1},m(\prev) + 1). \end{displaymath}
\begin{figure}[h!]
\begin{centering}
\includegraphics[scale = 0.7]{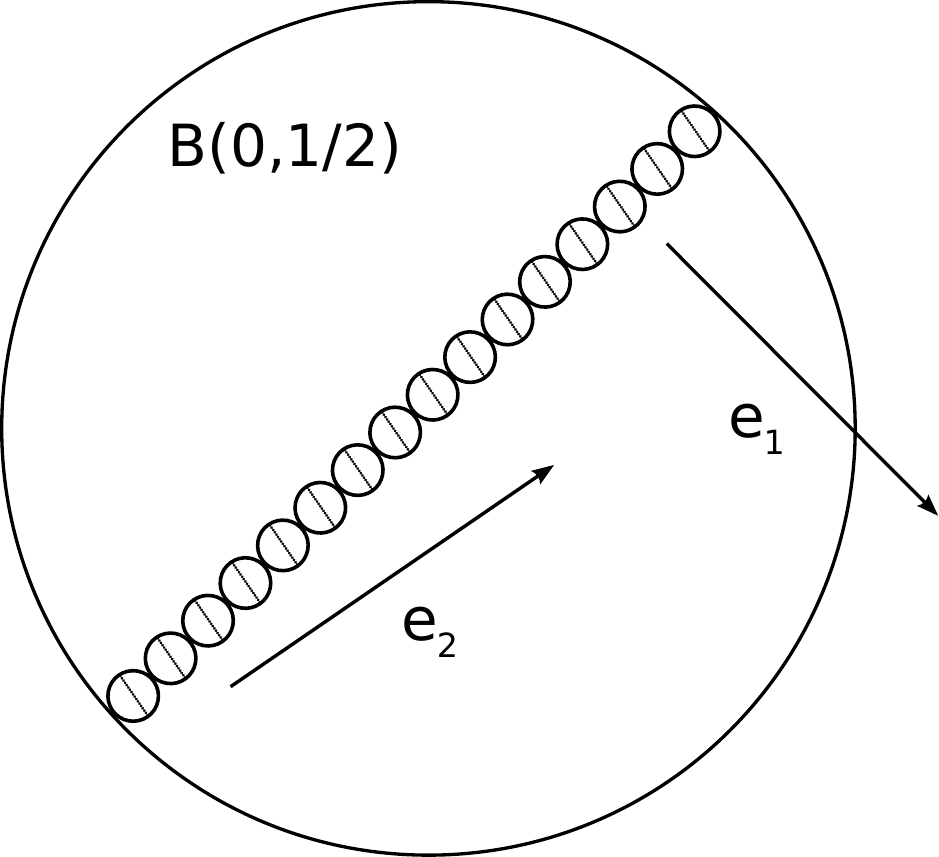}
\caption{A simultaneous depiction of $K(e_{1},1), K(e_{1},2)$ and $K(e_{2},1)$.}\label{fig1}
\end{centering}
\end{figure}
Figure \ref{fig1} shows the idea how to define the set $K(e_{m},n)$. Inside every one of the $p$ balls $B$, which constitute $K_{\prev}$, we place $q$ smaller balls on the diameter of $B$, which is perpendicular to $e_{m}$. Then the projection $\rho_{e_{m}}(K(e_{m},n))$ onto the line spanned by $e_{m}$ can be covered by $p$ intervals of of length $d/q$. The values of $p$ and $d$ only on $K_{\prev}$, whereas $q$ is a completely free parameter. We take $q$ so large that
\begin{displaymath} \calH^{s_{n}}_{1/n}(\rho_{e_{m}}[K(e_{m},n)]) \leq p \cdot \left(\frac{d}{q}\right)^{s_{n}} \leq \frac{1}{2}. \end{displaymath}
Then, we may choose $J(e_{m},n)$ to be an open arc centered at $e_{m}$ so small that
\begin{displaymath} \calH^{s_{n}}_{1/n}(\rho_{e}[K(e_{m},n)]) \leq 1, \qquad e \in J(e_{m},n). \end{displaymath}
Since $K \subset K(e_{m},n)$, this gives \eqref{form28} and completes the induction. The fact the set $K$ produced by the construction satisfies $\calH^{1}(K) > 0$ is standard: every ball in the 'previous generation' is replaced by a fairly uniformly distributed collection of (almost disjoint) new balls, and the sum of the diameters of the new balls equals the sum of the diameters of the previous balls. In fact, the construction of $K$ falls under the general scheme described in \cite[\S4.12]{Ma}, and the conclusion there is precisely that $0 < \calH^{1}(K) < \infty$.
\end{proof}

\section{The example in Theorem \ref{main2}}

In order to prove Theorem \ref{main2}, we will inductively and simultaneously construct Cantor type compact sets $K \subset \R^{2}$ and $E \subset S^{1}$ such that $\dim_{\p} K = \gamma \in [0,1]$, $\dim_{\p} E = 1$, and $\dim_{\p} K_{e} = \gamma/2$ for every direction $e \in E$. We first describe the construction of a 'generic exceptional set' $E \subset S^{1}$ with $\dim_{\p} E = 1$.  

\begin{construction}[The set $E$]\label{setE} Fix a sequence of numbers $(t_{j})_{j \in \N} \subset (0,1)$ such that $t_{j} \nearrow 1$ as $j \to \infty$. Let $(r_{j})_{j \in \Z_{+}}$ be a sequence of positive numbers, let $(n_{j})_{j \in \Z_{+}}$ be a sequence of natural numbers, and let $C_{j}^{\calI} \subset S^{1}$ be a collection of $\Gamma(j) := n_{0}n_{1} \cdots n_{j}$ points on the unit circle. Let $\calI_{j}$ be the collection of $\Gamma(j)$ closed arcs $I \subset S^{1}$ with midpoints in $C_{j}^{\calI}$ and length $\calH^{1}(I) = r_{j}$. We require the following properties from these items:
\begin{itemize} 
\item[(P0)] The values for $j = 0$ are $r_{0} = 1 = n_{0}$ and $C_{0} = \{(1,0)\}$. Hence, $\calI_{0}$ contains one arc of length one centered at the point $(1,0)$.
\item[(P1)] $r_{j} \searrow 0$ and $n_{j} \nearrow \infty$ as $j \to \infty$. Moreover, $n_{j} \nearrow \infty$ so quickly that 
\begin{equation}\label{nj} n_{j}^{1 - t_{j}}r_{j - 1}^{t_{j}} \geq 10, \qquad j \geq 1.  \end{equation} 
\item[(P2)] If $j \geq 1$, there are $n_{j}$ points of $C_{j}^{\calI}$ inside any arc $I \in \calI_{j - 1}$. The end-points of $I$ are not in $C_{j}^{\calI}$, the midpoint of $I$ is in $C_{j}^{\calI}$, and the points in $C_{j}^{\calI} \cap I$ are so evenly distributed that $d(x,y) > n_{j}^{-1}\calH^{1}(I)/10 = n_{j}^{-1}r_{j -1}/10$ for $x,y \in C_{j}^{\calI} \cap I$. 
\item[(P3)] If $j \geq 1$, the number $r_{j}$ is so small that for any $I \in \calI_{j - 1}$ the arcs in $\calI_{j}$ centered at the points in $C_{j}^{\calI} \cap I$ are disjoint and contained in $I$. 
\end{itemize} 
Now, suppose that we have chosen the numbers $n_{j}$ and $r_{j}$ and the sets $C_{j}^{\calI}$ so that properties (P0)--(P3) are in force. Then we define
\begin{displaymath} E := \bigcap_{j = 1}^{\infty} E_{j} := \bigcap_{j = 1}^{\infty} \bigcup_{I \in \calI_{j}} I. \end{displaymath}  
The sets $E_{j}$ are compact and non-empty and satisfy $E_{j} \supset E_{j + 1}$ by (P3), so $E$ is a non-empty compact subset of $S^{1}$. In order to evaluate $\dim_{\p} E$, first note that $C_{j}^{\calI} \subset E$ for any $j \geq 0$, by (P2). Next, let $U \subset S^{1}$ be an open set intersecting $E$. Then $U$ contains an arc $I \in \calI_{j - 1}$ for arbitrarily large indices $j \in \N$. This yields
\begin{displaymath} P\left(E \cap U,\frac{r_{j - 1}}{10n_{j}}\right) \geq P\left(C_{j}^{\calI} \cap I,\frac{r_{j - 1}}{10n_{j}}\right) \stackrel{\text{(P2)}}{\geq} n_{j} \stackrel{\text{(P1)}}{\geq} \frac{10}{10^{t_{j}}} \left(\frac{r_{j - 1}}{10n_{j}}\right)^{-t_{j}} \geq \left(\frac{r_{j - 1}}{10n_{j}}\right)^{-t_{j}}. \end{displaymath} 
Since $t_{j} \nearrow 1$ as $j \to \infty$, this shows that $\overline{\dim}_{\B} [E \cap U] = 1$ for any open set $U \subset S^{1}$ with $E \cap U \neq \emptyset$, and so $\dim_{\p} E = 1$ by Proposition \ref{falcProp}.
\end{construction}

The important fact here is that the choice of the numbers $r_{j}$ above is fairly arbitrary for $j \geq 1$: we may choose them as small as we wish, but, in light of \eqref{nj}, we will then have to compensate by choosing the numbers $n_{j}$ very large. The following auxiliary result will be used in constructing the examples in both Theorem \ref{main2} and \ref{bigEx}:

\begin{lemma}\label{grid} Let $G_{n} \subset \R^{2}$ be a set homothetic to the $n \times n$ grid $\{1,\ldots,n\} \times \{1,\ldots,n\} \subset \R^{2}$. Then, if $e \in S^{1}$ is the vector $e = c(1,pq^{-1}) \in S^{1}$, where $p,q \in \Z$ and $c = (1 + p^{2}q^{-2})^{-1/2}$, we have
\begin{displaymath} \card \rho_{e}(G_{n}) \leq (1 + p)(1 + q)n, \qquad n \in \N. \end{displaymath} 
\end{lemma}

\begin{proof} We may assume that $G_{n} = \{1,\ldots,n\} \times \{1,\ldots,n\}$, since any homothety $h(\bar{x}) = r\bar{x} + \bar{v}$ commutes with projections. If $t \in \rho_{e}(G_{n})$, find a point $(x,y) \in G_{n}$ such that $x + py/q = c^{-1}t$, and note that
\begin{displaymath} (x + kp) + \frac{p(y - kq)}{q} = \frac{t}{c}, \qquad k \in \Z. \end{displaymath} 
In particular, $\rho_{e}^{-1}\{t\} \supset \{(x,y) + k(p,-q) : 1 \leq k \leq n\}$. On the other hand, since $(x,y) \in G_{n}$, such points $(x,y) + k(p,-q)$ with $1 \leq k \leq n$ are contained in the product $\{1,\ldots,n(1 + p)\} \times \{-nq + 1,\ldots,n\} =: G_{n}'$, which has cardinality $n^{2}(1 + p)(1 + q)$. Now we have shown that for every $t \in \rho_{e}(G_{n})$ there exist at least $n$ points in the set $\rho_{e}^{-1}\{t\} \cap G_{n}'$. Since the pre-images $\rho_{e}^{-1}\{t\}$ are disjoint for various $t \in \R$, this yields the inequality
\begin{displaymath} n \cdot \card \rho_{e}(G_{n}) \leq \card G_{n}' = n^{2}(1 + p)(1 + q), \end{displaymath}
or $\card \rho_{e}(G_{n}) \leq n(1 + p)(1 + q)$, as claimed.
\end{proof}

\begin{proof}[Proof of Theorem \ref{main2}] The idea is to construct the set $K$ by an inductive procedure, and, in the process, choose the parameters of the 'generic' exceptional set $E$ so that (P0)--(P3) are satisfied, and $\dim_{\p} K_{e} = \gamma/2$ for every direction $e \in E$. The notation related to the construction of $E$ will be the same as in Construction \ref{setE}.

To construct $K$, we will define finite collections $\calQ_{j}$, $j \in \N$, of closed squares in $Q \subset \R^{2}$ of equal side-lengths $\ell(Q) =: \ell_{j}$ and write $K_{j} := \bigcup_{Q \in \calQ_{j}} Q$. The set $K$ is then be defined by $K = \bigcap_{j \in \N} K_{j}$. The collection of all midpoints of the squares in $\calQ_{j}$ is denoted by $C_{j}^{\calQ}$. Assume that $0 < \gamma < 2$, as we may, and fix a sequence $(\gamma_{j})_{j \in \N} \subset (0,2)$ such that $\gamma_{j} \nearrow \gamma$ as $j \to \infty$. We maintain the following invariants throughout the process of constructing the squares $\calQ_{j}$:
\begin{itemize}
\item[(i)] The collection $\calQ_{0}$ consists of only one square, namely $Q_{0} = [0,1]^{2}$. For every $j \geq 1$, we have $C^{\calQ}_{j - 1} \subset K_{j} \subset K_{j - 1}$.
\item[(ii)] For every $j \geq 0$, the collection $C^{\calI}_{j}$ consists of some points of the form $c(1,pq^{-1}) \in S^{1}$, where $p,q \in \Z$, $q \neq 0$, and $c = (1 + p^{2}q^{-2})^{-1/2}$. Moreover, $C^{\calI}_{j} \supset C^{\calI}_{j - 1}$ for every $j \geq 1$.
\item[(iii)] Whenever $j \geq 0$, $e \in I \in \calI_{j}$ and $\ell_{j} \leq l \leq 1$, we have
\begin{displaymath} N(\rho_{e}(K_{j}),l) \leq l^{-\gamma/2}. \end{displaymath} 
We also need the following technical hypothesis, which is only required for the induction to work: replace every square $Q \in \calQ_{j}$ by a smaller co-centric closed square $Q^{l}$ of side-length $l \leq \ell_{j}$ to obtain a new collection of squares $\calQ_{j}^{l}$, see Figure \ref{fig3}. Denote the union of these squares by $K_{j}^{l}$. Then
\begin{displaymath} N(\rho_{e}(K_{j}^{l}),l) \leq l^{-\gamma/2} \end{displaymath}
for all $e \in I \in \calI_{j}$.
\item[(iv)] If $j \geq 1$ and $Q \in \calQ_{j - 1}$, then $P(C_{j}^{\calQ} \cap Q,\ell_{j}/2) \geq \ell_{j}^{-\gamma_{j}}$. 
\end{itemize}
Let us now initiate the induction. Condition (i) forces us to choose $\calQ_{0} = \{Q_{0}\} = \{[0,1]^{2}\}$, $C_{0}^{\calQ} = \{(1/2),1/2)\}$ and $\ell_{0} = 1$. It is clear that properties (i)--(iii) are satisfied for these parameters, and (iv) says nothing at this point. In particular, the 'technical hypothesis' in (iii) is satisfied, since the set $K_{0}^{\ell}$ is nothing but a single square of side-length $\ell < 1$. Also, recall that $n_{0} = 1 = r_{0}$ and $C^{\calI}_{0} = \{(1,0)\} \in S^{1}$ according to (P0) of Construction \ref{setE}.

Next, let us assume that $j \geq 1$ and $\calQ_{j - 1}$, $C^{\calQ}_{j - 1}$, $\ell_{j - 1}$, $\calI_{j - 1}$, $n_{j - 1}$, $r_{j - 1}$ and $C^{\calI}_{j - 1}$ have already been defined so that (i)--(iii) hold. We will now describe how to define the parameters corresponding to the index $j$, so that all the conditions (i)--(iii) are satisfied (we exclude (iv), because \textbf{assuming} property (iv) for the index $j - 1$ is not necessary to acquire it for the index $j$). First, choose $n_{j}$ so large that (P1) in Construction \ref{setE} is satisfied, that is, $n_{j}^{1 - t_{j}}r_{j - 1}^{t_{j}} \geq 10$. Then, inside every interval $I \in \calI_{j - 1}$, place $n_{j}$ points of the form $c(1,pq^{-1})$, $p,q \in \Z$, $q \neq 0$, $c = (1 + p^{2}q^{-2})^{-1/2}$, so that the endpoints of $I$ are excluded and the midpoint of $I$ is included in the selection (this is possible by (ii)), and so that the mutual distance of any pair of these points is at least $n_{j}^{-1}\calH^{1}(I)/10 = n_{j}^{-1}r_{j - 1}/10$. Points of the correct form are dense on $S^{1}$, so the existence of such a selection is no issue -- as far as we are not interested in how large $p$ and $q$ can get. The collection of all such points, for every interval $I \in \calI_{j - 1}$, is the new midpoint set $C^{\calI}_{j}$. Now (P2) is satisfied.
\begin{figure}[h!]
\begin{centering}
\includegraphics[scale = 0.7]{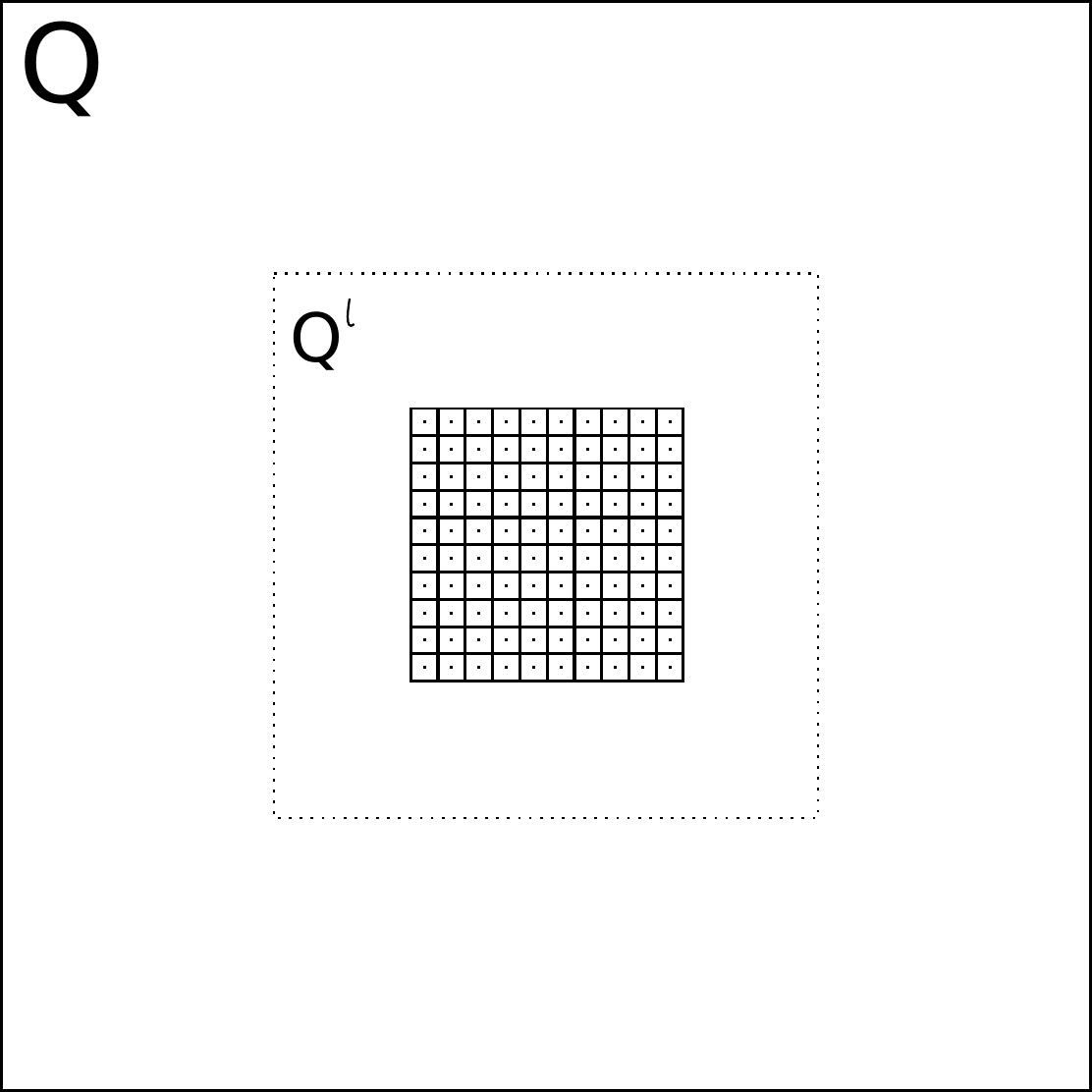}
\caption{A square $Q \in \calQ_{j - 1}$, its children in $\calQ_{j}$, and a square of the form $Q^{l} \in \calQ_{j - 1}^{l}$, as in the 'technical hypothesis' of (iii).}\label{fig3}
\end{centering}
\end{figure}

Next, we will define $\ell_{j}$ and the collection $\calQ_{j}$. Write $q_{j - 1} := \card \calQ_{j - 1}$ and $M_{j} := \max\{(1 + p)(1 + q) : c(1,pq^{-1}) \in C^{\calI}_{j}\}$. Choose $\ell_{j}$ so small that
\begin{equation}\label{ell} \ell_{j}^{1 - \gamma_{j}/2} < \ell_{j - 1} \quad \text{and} \quad q_{j - 1} \cdot \ell_{j}^{(\gamma - \gamma_{j})/2} \leq \min\left\{\frac{1}{4},\frac{1}{10M_{j}}\right\}. \end{equation}
These choices can clearly be made so that $\ell_{j}^{-\gamma_{j}/2}$ is an integer, and so $\ell_{j}^{-\gamma_{j}}$ is the square of an integer. Now, inside each square $Q \in \calQ_{j - 1}$, place $\ell_{j}^{-\gamma_{j}}$ squares of side-length $\ell_{j}$, so that the \textbf{union} of the new squares also forms a square $Q'$ of side-length $\ell(Q') = \ell_{j} \cdot \ell_{j}^{-\gamma_{j}/2} = \ell_{j}^{1 - \gamma_{j}/2} <  \ell_{j - 1} = \ell(Q)$, and the midpoint of $Q$ coincides with the midpoint of $Q'$, see Figure \ref{fig3}. The collection $\calQ_{j}$ then consists of all the $q_{j} := q_{j - 1} \cdot \ell^{-\gamma_{j}}$ small squares (of side-length $\ell_{j}$) so obtained, for every choice of $Q \in \calQ_{j - 1}$. To prove (iv), simply note that a packing of the new midpoint set $C_{j}^{\calQ}$ intersected with any square $Q \in \calQ_{j - 1}$ is obtained by placing a ball of radius $\ell_{j}/2$ centered at every point in $C_{j}^{\calQ} \cap Q$. This yields
\begin{displaymath} P(C_{j}^{\calQ} \cap Q,\ell_{j}/2) \geq \ell_{j}^{-\gamma_{j}}, \end{displaymath}  
which is (iv).

It remains to define $r_{j}$ (and $\calI_{j}$, of course, but this is completely determined by $r_{j}$ and $C^{\calI}_{j}$) and prove (iii). Set
\begin{displaymath} r_{j} := \frac{1}{2 \cdot (4q_{j}^{2})^{1/\gamma}}. \end{displaymath}
We start by proving the 'technical hypothesis' of (iii). Note that $C_{j}^{\calQ} \cap Q$ is a grid homothetic to $\{1,\ldots,\ell_{j}^{-\gamma_{j}/2}\} \times \{1,\ldots,\ell_{j}^{-\gamma_{j}/2}\}$, for any $Q \in \calQ_{j - 1}$ Hence, if $e  = c(1,pq^{-1}) \in C_{j}^{\calI}$, the previous lemma shows that 
\begin{equation}\label{form3} \card \rho_{e}(C_{j}^{\calQ}) \leq q_{j - 1}(1 + p)(1 + q)\ell_{j}^{-\gamma_{j}/2} \leq M_{j} \cdot q_{j - 1} \cdot \ell_{j}^{-\gamma_{j}/2} \stackrel{\eqref{ell}}{\leq} \frac{\ell_{j}^{-\gamma/2}}{10}. \end{equation}
Now, fix $l \leq \ell_{j}$ and consider the squares $\calQ_{j}^{l} := \{Q^{l} : Q \in \calQ_{j}\}$ as defined in (iii). Recall that these are the squares cocentric with the squares in $\calQ_{j}$ but with side-length only $l$. The $\rho_{e}$-projection of the set $K_{j}^{l} = \bigcup_{Q \in \calQ_{j}} Q^{l}$ consists of intervals of length at most $\sqrt{2} \cdot l$ with midpoints in the set $\rho_{e}(C_{j}^{\calQ})$. Hence, we may infer from \eqref{form3} that
\begin{displaymath}  N(\rho_{e}(K_{j}^{l}),l) \leq \frac{\ell_{j}^{-\gamma/2}}{5} \leq \frac{l^{-\gamma/2}}{5}, \qquad e \in C_{j}^{\calI}. \end{displaymath} 
Next, note that if $\xi \in B(e,l) \cap S^{1}$ with $e \in C_{j}^{\calI}$, we still have
\begin{displaymath} N(\rho_{\xi}(K_{j}^{l}),l) \leq l^{-\gamma/2}, \end{displaymath} 
since the intervals of length no more than $\sqrt{2} \cdot l$ that make up $\rho_{\xi}(K_{j}^{l})$ are certainly covered by the intervals that constitute $\rho_{e}(K_{j}^{l})$, stretched by a factor of five. In particular, this shows that $N(\rho_{\xi}(K^{l}_{j}),l) \leq l^{-\gamma/2}$, whenever $l \geq 2r_{j}$ and $\xi \in I \in \calI_{j}$ (then $\xi$ is at distance no more than $2r_{j} \leq l$ from one of the points in $C_{j}^{\calI}$). On the other hand, if $l \leq 2r_{j} = 1/(4q_{j}^{2})^{1/\gamma}$, we have the trivial estimate
\begin{displaymath} N(\rho_{\xi}(K_{j}^{l}),l) \leq 2q_{j} \leq l^{-\gamma/2}, \end{displaymath}
which follows from the fact that $\rho_{\xi}(K_{j}^{l})$ is the union of $q_{j}$ intervals of length no more than $\sqrt{2} \cdot l$. This proves the 'technical hypothesis' of (iii).

Finally, it is time to prove the first part of (iii). Fix $l \in [\ell_{j},1]$ and $e \in I \in \calI_{j}$. If $l > \ell_{j - 1}$, we simply note that $e \in J$ for some $J \in \calI_{j - 1}$ and use the induction hypothesis in (i) and (iii) to conclude that
\begin{displaymath} N(\rho_{e}(K_{j}),l) \leq N(\rho_{e}(K_{j - 1}),l) \leq l^{-\gamma/2}. \end{displaymath}
Next, recall that the squares of $\calQ_{j}$ inside any fixed square $Q \in \calQ_{j - 1}$ are arranged so that they form a square $Q'$ of side-length $\calL_{j} := \ell_{j}^{1 - \gamma_{j}/2} < \ell_{j - 1}$, which has the same center as $Q$. For any $l \in [\calL_{j},\ell_{j - 1}]$, we then note that the union of these squares $\{Q' : Q \in \calQ_{j - 1}\}$ is contained in the union $K_{j - 1}^{l}$ of the squares $\calQ_{j - 1}^{l} = \{Q^{l} : Q \in \calQ_{j - 1}\}$, as defined in the 'technical hypothesis' of (iii). This means that
\begin{displaymath} N(\rho_{e}(K_{j}) ,l) \leq N(\rho_{e}(K_{j - 1}^{l}),l) \leq l^{-\gamma/2}, \qquad e \in I \in \calI_{j - 1}, \end{displaymath}
by the induction hypothesis. In particular, this holds for $e \in I \in \calI_{j}$. We are left with the case $l \in [\ell_{j},\calL_{j}]$. The projection $\rho_{e}(K_{j})$ in \textbf{any} direction $e \in S^{1}$ is the union of $q_{j - 1} = \card \calQ_{j - 1}$ intervals of length no more than $\sqrt{2} \cdot \calL_{j}$. Since $l \leq \calL_{j}$, such a union can be covered by $4q_{j - 1} \cdot \calL_{j}/l$ intervals of length $l$. This and \eqref{ell} yields the estimate
\begin{align*} N(\rho_{e}(K_{j}),l) \cdot l^{\gamma/2} & \leq 4q_{j - 1} \cdot \frac{\calL_{j}}{l} \cdot l^{\gamma/2}\\
& = 4q_{j - 1} \cdot \ell_{j}^{1 - \gamma_{j}/2} \cdot l^{\gamma/2 - 1}\\
& \leq 4q_{j - 1} \cdot \ell_{j}^{(\gamma - \gamma_{j})/2} \stackrel{\eqref{ell}}{\leq} 1.  \end{align*} 
The proof of (iii) is finished. This completes the inductive step and the construction of the sets $E \subset S^{1}$ and $K \subset [0,1]^{2}$. The construction of the set $E$ abides by the scheme in Construction \ref{setE}, so we have $\dim_{\p} E = 1$. It only remains to verify that $\dim_{\p} K = \gamma$ and $\dim_{\p} K_{e} \leq \gamma/2$ for every direction $e \in E$. All the midpoint sets $C_{j}^{\calQ}$ are contained in $K$ by (i), so (iv) combined with Proposition \ref{falcProp} gives $\dim_{\p} K = \gamma$. If $e \in E$, then, for all $j \in \N$, we have $e \in I$ for some arc $I \in \calI_{j}$. Now we may deduce from (iii) that
\begin{displaymath} \dim_{\p} K_{e} \leq \overline{\dim}_{\B} K_{e} \leq \limsup_{l \to 0} \frac{\log N(K_{e},l)}{- \log l} \leq \frac{\gamma}{2}. \end{displaymath}
This completes the proof of Theorem \ref{main2}.
\end{proof} 

\section{Proof of Theorem \ref{packingCategory}}

The proof of Theorem \ref{packingCategory} divides into three parts. First, we reduce the situation from analytic sets to compact sets using a lemma from \cite{FH}. Second, we make a further reduction showing that it suffices to prove Theorem \ref{packingCategory} for upper box dimension $\overline{\dim}_{\B}$ instead of packing dimension $\dim_{\p}$. Third, we prove Theorem \ref{packingCategory} for upper box-dimension using a simple combinatorial approach.

Let $A \subset \R^{2}$ be a compact set, and let
\begin{displaymath} \m := \sup\{\dim_{\p} A_{e} : e \in S^{1}\}. \end{displaymath}
We will establish Theorem \ref{packingCategory} by showing that
\begin{equation}\label{form29} \underline{\dim}_{\MB} \left\{e \in S^{1} : \dim_{\p} A_{e} < \sigma \right\} \leq 1 + \sigma - \m, \qquad 0 \leq \sigma \leq \m, \end{equation}
where $\underline{\dim}_{\MB}$ denotes the \emph{lower modified box-dimension}
\begin{displaymath} \underline{\dim}_{\MB} B := \inf \left\{\sup_{j} \underline{\dim}_{\B} F_{j} : B \subset \bigcup_{j \in \N} F_{j}\right\}, \end{displaymath}
and $\underline{\dim}_{\B}$ is the lower box-dimension
\begin{displaymath} \underline{\dim}_{\B} F := \liminf_{\delta \to 0} \frac{\log N(F,\delta)}{-\log \delta}. \end{displaymath}
We may infer from \eqref{form29} and the definition of $\m$ that the set $\{e \in S^{1} : \dim_{\p} A_{e} \neq \m\}$ has zero length. Moreover, if we manage to prove \eqref{form29} for $\sigma < \m$, we can, by definition of $\underline{\dim}_{\MB}$, cover the set $\{e \in S^{1} : \dim_{\p} A_{e} < \sigma\}$ with countably many sets $F_{j}$ with $\underline{\dim}_{\B} F_{j} < 1$. The sets $F_{j}$ are nowhere dense, so $\{e \in S^{1} : \dim_{\p} A_{e} < \sigma\}$ is meagre by definition. The set $\{e \in S^{1} : \dim_{\p} A_{e} < \m\}$ is then meagre as well.

\subsection{First reduction} We cite the planar version of \cite[Lemma 7]{FH}.
\begin{lemma}\label{FHLemma} Let $0 \leq t < 1$, let $e \in S^{1}$, and let $A \subset \R^{2}$ be an analytic set such that $0 \leq t < \dim_{\p} A_{e}$. Then there exists a compact set $K \subset A$ with $t < \dim_{\p} K_{e}$.
\end{lemma} 
It follows immediately that it suffices to prove the bound \eqref{form29} for compact sets only. Namely, if $A \subset \R^{2}$ is an analytic set with $\m = \m(A) > 0$, we may use Lemma \ref{FHLemma} to find a compact set $K \subset A$ with $m := \sup\{\dim_{\p} K_{e} : e \in S^{1}\}$ arbitrarily close to $\m$. Then
\begin{displaymath} \underline{\dim}_{\MB} \{e \in S^{1} : \dim_{\p} A_{e} < \sigma\} \leq \underline{\dim}_{\MB} \{e \in S^{1} : \dim_{\p} K_{e} < \sigma\} \leq 1 + \sigma - m \end{displaymath}
for $0 \leq \sigma \leq m$, assuming that we have \eqref{form29} for all compact sets. Letting $m \nearrow \m$ gives \eqref{form29} for $A$.

\subsection{Second reduction} Assume that we know how to prove the following.
\begin{thm}\label{boxCategory} Let $A \subset \R^{2}$ be a set, and let
\begin{displaymath} \m_{\B} := \sup \{\overline{\dim}_{\B} A_{e} : e \in S^{1}\}. \end{displaymath}
Then
\begin{equation}\label{form39} \underline{\dim}_{\MB} \{e \in S^{1} : \overline{\dim}_{\B} A_{e} < \sigma\} \leq 1 + \sigma - \m_{\B}, \qquad 0 \leq \sigma \leq \m_{\B}. \end{equation}
\end{thm}
Then, we claim that we can also prove \eqref{form29} \textbf{for compact sets}. This reduction is based on the following lemma, which will also be useful in the next section.
\begin{lemma}\label{packingIsBox} Let $\Dim$ be any countably stable notion of dimension, and let $\sigma,\beta > 0$. Suppose that there exist a Borel regular measure $\mu$ and a $\mu$-measurable set $B \subset \R^{2}$ such that $\mu(B) > 0$, and
\begin{displaymath} \Dim \{e \in S^{1} : \dim_{\p} B_{e} < \sigma\} > \beta. \end{displaymath}
Then there exists a compact set $K \subset B$ with $\mu(K) > 0$ such that 
\begin{displaymath} \Dim \{e \in S^{1} : \overline{\dim}_{\B} K_{e} < \sigma\} > \beta. \end{displaymath}
\end{lemma}

\begin{proof} Take a compact set $\tilde{K} \subset B$ such that $\mu(\tilde{K} \cap U) > 0$ for all open sets $U \subset \R^{2}$ which intersect $\tilde{K}$: any compact set $\tilde{K} \subset B \cap \spt \mu$ with positive measure will do, and such sets exists by \cite[Theorem 1.10(1)]{Ma}. Next, let $(U_{j})_{j \in \N}$ be the countable collection of all open balls with rational centers and rational radii that intersect $\tilde{K}$. Write $E := \{e \in S^{1} : \dim_{\p} B_{e} < \sigma\}$, and set 
\begin{displaymath} E_{j} := \{e \in S^{1} : \overline{\dim}_{\B} [\tilde{K} \cap \overline{U}_{j}]_{e} < \sigma\}. \end{displaymath}
Here $[\tilde{K} \cap \overline{U}_{j}]_{e} := \rho_{e}(\tilde{K} \cap \overline{U}_{j})$, as usual. We claim that $E \subset \bigcup_{j} E_{j}$. Let $e \in E$. Then $\dim_{\p} \tilde{K}_{e} < \sigma$, which, by definition, means that
\begin{displaymath} \inf \left\{\sup_{i} \overline{\dim}_{\B} F_{i} : \tilde{K}_{e} \subset \bigcup_{i \in \N} F_{i} \right\} < \sigma,  \end{displaymath} 
where the sets $F_{i}$ can be assumed to be closed. Now, let $(F_{i})_{i \in \N}$ be a countable collection of closed sets such that $\tilde{K}_{e} \subset \bigcup_{i} F_{i}$ and $\overline{\dim}_{\B} F_{i} < \sigma$ for every $i \in \N$. Since $\tilde{K}_{e}$ is compact, Baire's theorem tells us that some intersection $\tilde{K}_{e} \cap F_{i}$ must have interior points in the relative topology of $\tilde{K}_{e}$: in other words, we may find an open set $V \subset \R$ such that $\emptyset \neq \tilde{K}_{e} \cap V \subset F_{i}$. Since the open set $\rho_{e}^{-1}(V) \subset \R^{2}$ intersects $\tilde{K}$, we may deduce that the closure of one of the balls $U_{j}$ lies in $\rho_{e}^{-1}(V)$. Then
\begin{displaymath} \overline{\dim}_{\B} [\tilde{K} \cap \overline{U}_{j}]_{e} \leq \overline{\dim}_{\B} [\tilde{K}_{e} \cap V] \leq \overline{\dim}_{\B} F_{i} < \sigma \end{displaymath}
which means that $e \in E_{j}$. Since $\Dim$ is countably stable, we may now conclude that
\begin{displaymath} \beta < \Dim E \leq \sup_{j} \Dim E_{j}. \end{displaymath}
Thus, one of the sets $E_{j}$ must satisfy $\Dim E_{j} > \beta$. Now $K = \tilde{K} \cap \overline{U}_{j}$, for the same index $j$, is the set we were after.
\end{proof}

Let us see how to prove \eqref{form29} for a compact set $K \subset \R^{2}$, assuming Theorem \ref{boxCategory}. Suppose that \eqref{form29} fails. Then there exist numbers $\sigma < m < \m$ such that
\begin{equation}\label{form40} \underline{\dim}_{\MB} \{e \in S^{1} : \dim_{\p} K_{e} < \sigma\} > 1 + \sigma - m. \end{equation}
Pick a direction $\xi \in S^{1}$ such that $\dim_{\p} K_{\xi} > m$. Then, according to a result of Joyce and Preiss \cite{JP}, we may extract a compact subset $R \subset K_{\xi}$ with $0 < \calP^{m}(R) < \infty$. Note that $\rho_{\xi} \colon K \cap \rho_{\xi}^{-1}(R)\to R$ is a continuous surjection between compact spaces, so we may use \cite[Theorem 1.20]{Ma} to find a measure $\mu$ supported on $K \cap \rho_{\xi}^{-1}(R)$ such that 
\begin{equation}\label{form38} \rho_{\xi\sharp}\mu = \calP^{m}\llcorner R. \end{equation}
We then apply Lemma \ref{packingIsBox} with the choices $\Dim = \underline{\dim}_{\MB}$, $B = K \cap \rho_{\xi}^{-1}(R)$, and the measure $\mu$ we just constructed. Since $\mu(B) > 0$ and \eqref{form40} holds, we may extract a compact set $K^{\mu} \subset B = K \cap \rho_{\xi}^{-1}(R)$ with $\mu(K^{\mu}) > 0$ such that
\begin{equation}\label{form41} \underline{\dim}_{\MB} \{e \in S^{1} : \overline{\dim}_{\B} K^{\mu}_{e} < \sigma\} > 1 + \sigma - m. \end{equation}
Recalling \eqref{form38}, we have
\begin{displaymath} \calP^{m}(K^{\mu}_{\xi}) = \mu[\rho_{\xi}^{-1}(K_{\xi}^{\mu})] \geq \mu(K^{\mu}) > 0, \end{displaymath}
which certainly implies that $\overline{\dim}_{\B} K_{\xi}^{\mu} \geq m$. In particular, we may infer from Theorem \ref{boxCategory} that
\begin{displaymath} \underline{\dim}_{\MB} \{e \in S^{1} : \overline{\dim}_{\B} K^{\mu}_{e} < \sigma\} \leq 1 + \sigma - m. \end{displaymath}
This contradicts \eqref{form41} and completes the second reduction.

\subsection{Proof of Theorem \ref{boxCategory}} We first introduce a discrete notion of 'well-spread $\delta$-separated sets' and prove a version of Marstrand's projection theorem for such sets. Then, we derive Theorem \ref{boxCategory} by finding large well-spread sets inside the given arbitrary set $A$.
\begin{definition}\label{deltaSet} A finite set $C \subset B(0,1)$ is called \emph{a $(\delta,1)$-set}, if the points in $C$ are $\delta$-separated, and
\begin{displaymath} \card [C \cap B(x,r)] \lesssim \frac{r}{\delta}, \qquad x \in \R^{2},\: r \geq \delta. \end{displaymath}
\end{definition}
\begin{proposition}\label{discreteMarstrand} Let $C \subset \R^{2}$ be a $(\delta,1)$-set with $n \in \N$ points. Let $\tau > 0$, and let $E \subset S^{1}$ be a $\delta$-separated collection of vectors such that
\begin{displaymath} N(C_{e},\delta) \leq \delta^{\tau}n, \qquad e \in E. \end{displaymath}
Then $\card E \lesssim \delta^{\tau - 1} \cdot \log (1/\delta)$.
\end{proposition}

\begin{proof} Given $e \in E$, define the family of sets $\calT_{e}$ as follows:
\begin{displaymath} \calT_{e} := \{\rho_{e}^{-1}[j\delta,(j + 1)\delta) : j \in \Z\}. \end{displaymath}
Thus, $\calT_{e}$ consists of $\delta$-tubes perpendicular to $e$. Define the relation $\sim_{e}$ on $C \times C$ by 
\begin{displaymath} x \sim_{e} y \quad \Longleftrightarrow \quad x,y \in T \in \calT_{e}. \end{displaymath}
Thus, the points $x$ and $y$ are required to lie in a common $\delta$-tube in $\calT_{e}$. Let
\begin{displaymath} \mathcal{E} := \sum_{e \in E} \card \{(x,y) \in C \times C : x \sim_{e} y\}.  \end{displaymath}
If $x,y \in C$, it is a simple geometric fact that there can be no more than $\lesssim |x - y|^{-1}$ directions in $E$ such that $x \sim_{e} y$. This gives the upper bound
\begin{align*} \mathcal{E} & = \sum_{x \in C} \sum_{j : \delta \leq 2^{j} \leq 1} \mathop{\sum_{y \in C}}_{ 2^{j} \leq |x - y| < 2^{j + 1}} \card \{e \in E : x \sim_{e} y\}\\
& \lesssim \sum_{x \in C} \sum_{j : \delta \leq 2^{j} \leq 1} \mathop{\sum_{y \in C}}_{ 2^{j} \leq |x - y| < 2^{j + 1}} |x - y|^{-1}\\
& \lesssim \sum_{x \in C} \sum_{j : \delta \leq 2^{j} \leq 1} \card [C \cap B(x,2^{j + 1})] \cdot 2^{-j}\\
& \lesssim \delta^{-1} \cdot \sum_{x \in C} \sum_{j : \delta \leq 2^{j} \leq 1} 2^{j} \cdot 2^{-j} \asymp \delta^{-1} \cdot n \cdot \log\left(\frac{1}{\delta}\right).
\end{align*}
Next, let us try to find a lower bound for $\mathcal{E}$ in terms of $\card E$. Let $e \in E$. We may and will assume that $\delta^{\tau}n \geq 1$. Since $N(C_{e},\delta) \leq \delta^{\tau}n$, we find that $C$ can be covered by some tubes $T_{1},\ldots,T_{K} \in \calT_{e}$, where $K \lesssim \delta^{\tau}n$. This gives
\begin{align*} \card \{(x,y) \in C \times C : x \sim_{e} y\} & = \sum_{j = 1}^{K} \card \{(x,y) \in C \times C : x,y \in T_{j}\}\\
& = \sum_{j = 1}^{K} \card [C \cap T_{j}]^{2}\\
& \stackrel{\text{C-S}}{\geq} \frac{1}{K} \cdot \left(\sum_{j = 1}^{K} \card [C \cap T_{j}]\right)^{2}\\
& \gtrsim \delta^{-\tau} \cdot n^{-1} \cdot (\card C)^{2} = \delta^{-\tau} \cdot n.
\end{align*}
The letters C-S refer to Cauchy-Schwarz. This immediately yields 
\begin{displaymath} \delta^{-\tau} \cdot n \cdot \card E \lesssim \mathcal{E} \lesssim \delta^{-1} \cdot n \cdot \log \left(\frac{1}{\delta}\right), \end{displaymath}
and the asserted bound follows.
\end{proof}

\begin{proof}[Proof of Theorem \ref{boxCategory}] Recall that $A \subset \R^{2}$ is an arbitrary set, and
\begin{displaymath} \m_{\B} = \sup \{\overline{\dim}_{\B} A_{e} : e \in S^{1}\}. \end{displaymath}
Let $0 \leq \sigma < \m_{\B}$ and write $\tilde{E} := \{e \in S^{1} : \overline{\dim}_{\B} A_{e} < \sigma\}$. We observe that
\begin{displaymath} \tilde{E} \subset \bigcup_{i \in \N} \bigcap_{\delta \in (0,1/i)} \{e \in S^{1} : N(A_{e},\delta) \leq \delta^{-\sigma}\} =: \bigcup_{i \in \N} E_{i}, \end{displaymath}
whence, by definition of $\underline{\dim}_{\MB}$, it suffices to prove that
\begin{equation}\label{form31} \sup_{i} \underline{\dim}_{\B} E_{i} \leq 1 + \sigma - \m_{\B}. \end{equation}
Fix $i \in \N$ and let $E := E_{i}$. Also, fix $\sigma < m < \m_{\B}$, and choose a direction $\xi \in S^{1}$ such that $N(A_{\xi},\delta) \geq \delta^{-m}$ for arbitrarily small values of $\delta > 0$. Choose some such value $\delta$, and use the information $N(A_{\xi},\delta) \geq \delta^{-m}$ to find a $\delta$-separated set $C_{\delta} \subset A_{\xi}$ of cardinality $\card C_{\delta} \gtrsim \delta^{-m}$. Write $\calT_{\xi}$ for the same family of tubes in $\R^{2}$ as in the previous proof. Since $C_{\delta} \subset A_{\xi}$, there exist tubes $T_{1},\ldots,T_{K} \in \calT_{\xi}$ such that 
\begin{itemize}
\item[(a)] the tubes are at least $\delta$-separated from one another,
\item[(b)] $K \gtrsim \delta^{-m}$, and 
\item[(c)] every tube $T_{j}$ contains a point $x_{j} \in A$.
\end{itemize}
The set $C^{\delta} := \{x_{j} : 1 \leq j \leq K\} \subset A$ is clearly $\delta$-separated, and $n := \card C^{\delta} \gtrsim \delta^{-m}$. More importantly, $C^{\delta}$ is a $(\delta,1)$-set. This is a direct consequence of the fact that any ball $B(x,r) \subset \R^{2}$ of radius $r \geq \delta$ intersects no more than $\lesssim r/\delta$ tubes in $\calT_{\xi}$. The previous proposition now implies that
\begin{displaymath} N(\{e : N(C^{\delta}_{e},\delta) \leq \delta^{-\sigma}\},\delta) \leq N(\{e : N(C_{e}^{\delta},\delta) \leq \delta^{m - \sigma}n\},\delta) \leq \delta^{m - \sigma - 1 - \varepsilon} \end{displaymath}
for $\delta > 0$ sufficiently small. Since $C^{\delta} \subset A$, we have
\begin{displaymath} E = \bigcap_{\delta \in (0,1/i)} \{e : N(A_{e},\delta) \leq \delta^{-\sigma}\} \subset \{e : N(C^{\delta}_{e},\delta) \leq \delta^{-\sigma}\}, \end{displaymath}
so we have found arbitrarily small values of $\delta > 0$ such that $N(E,\delta) \leq \delta^{m - \sigma - 1 - \varepsilon}$. This gives \eqref{form31} and completes the proof.
\end{proof}

\begin{remark} We did not include the assertions $\m,\m_{\B} \geq 2\gamma/(2 + \gamma)$, see Remark \ref{FHRemark}, in the statements of Theorems \ref{packingCategory} and \ref{boxCategory}, because they are well-known, and combinatorial-geometric proofs already exist in \cite{FH1}. To see how the bounds would follow from our method, let us sketch the proof of $\m_{\B} \geq 2\gamma/(2 + \gamma)$ for any set $A \subset \R^{2}$ with $\overline{\dim}_{\B} A = \gamma \in (0,2]$. First of all, there exist arbitrarily small scales $\delta > 0$ such that $A$ contains a $\delta$-separated subset $C^{\delta}$ of cardinality between $\delta^{-\gamma + \varepsilon}$ and $\delta^{-\gamma}$. Then, it is easy to check that $C^{\delta}$ is, in fact, a $(\delta^{(2 + \gamma)/2},1)$-set, so Proposition \ref{discreteMarstrand} shows that $N(C^{\delta}_{e},\delta^{(2 + \gamma)/2}) \gtrsim \delta^{-\gamma + 2\varepsilon}$ for all but a very few ($\delta^{(2 + \gamma)/2}$-separated) directions. For all the 'good' directions we have 
\begin{displaymath} \frac{\log N(K_{e},\delta^{(2 + \gamma)/2})}{-\log \delta^{(2 + \gamma)/2}} \gtrsim \frac{\gamma - 2\varepsilon}{(2 + \gamma)/2} \approx \frac{2\gamma}{2 + \gamma}, \end{displaymath}
which means precisely that $\overline{\dim}_{\B} K_{e} \geq 2\gamma/(2 + \gamma)$ in almost every direction. To get the same conclusion for $\dim_{\p}$ instead of $\overline{\dim}_{\B}$, one has to pass through Lemma \ref{packingIsBox} in a similar spirit as we did in the second reduction.
\end{remark}

\section{Proofs of Theorems \ref{estimate} and \ref{estimate2}} 

The proof of Theorem \ref{estimate} is based on a modification of the argument we used in Proposition \ref{discreteMarstrand}. In the proof of Theorem \ref{estimate2}, the same structure is again present, but we also make use of a 'dimension conservation principle' due to H. Furstenberg.

\begin{proof}[Proof of the first estimate in Theorem \ref{estimate}] Frostman's lemma for analytic sets, see \cite{Ca}, and Lemma \ref{packingIsBox} combined reduce our task to proving the following assertion: assume that $\gamma \in (0,1)$, let $K \subset B(0,1)$ be a compact set supporting a Borel probability measure $\mu$ with $I_{\gamma}(\mu) < \infty$, and let $0 < \sigma < \gamma$. Then  the packing dimension of the exceptional set
\begin{displaymath} \tilde{E} := \{e \in S^{1} : \overline{\dim}_{\B} K_{e} < \sigma\} \end{displaymath}
admits the estimate
\begin{displaymath} \dim_{\p} \tilde{E} \leq \frac{\sigma\gamma}{\gamma + \sigma(\gamma - 1)}. \end{displaymath}
As in the previous section, we note that $\tilde{E}$ satisfies
\begin{displaymath} \tilde{E} \subset \bigcup_{i \in \N} \bigcap_{\delta \in (0,1/i)} \{e \in S^{1} : N(K_{e},\delta) \leq \delta^{-\sigma}\} =: \bigcup_{i \in \N} E_{i}. \end{displaymath}
So, it suffices to prove that
\begin{equation}\label{form37} \overline{\dim}_{\B} E_{i} \leq \frac{\sigma\gamma}{\gamma + \sigma(\gamma - 1)} \end{equation}
for every $i \in \N$. Fix $i \in \N$, $0 < \delta < 1/i$, and write $E := E_{i}$. Let us redefine some of the notation from the previous section. There will be tubes: given $e \in S^{1}$, we write
\begin{displaymath} \calT_{e} = \{\rho_{e}^{-1}[j\delta^{\rho},(j + 1)\delta^{\rho} : j \in \Z\}, \end{displaymath}
where $\rho = \rho(\sigma,\gamma) \geq 1$ is a parameter to be chosen later. We define the relation $\sim_{e}$ as before:
 \begin{displaymath} x \sim_{e} y \quad \Longleftrightarrow \quad x,y \in T \in \calT_{e}. \end{displaymath}
 Let $E_{0} \subset E$ be any $\delta$-separated finite subset. This time, the energy $\mathcal{E}$ looks like 
 \begin{displaymath} \mathcal{E} := \sum_{e \in E_{0}} \mu \times \mu(\{(x,y) : x \sim_{e} y\}). \end{displaymath}
 We first aim to bound $\mathcal{E}$ from above. To this end, we make the \emph{a priori} assumption $M := \card E_{0} \lesssim \delta^{-\tau}$ for some $\tau \in (0,1]$. Of course, this is always satisfied with $\tau = 1$. Also, we need the simple geometric fact that the set $\{e \in S^{1} : x \sim_{e} y\}$ is an arc of length $\lesssim \delta^{\rho}/|x - y|$. Thus, there are no more than $\lesssim \max\{1,\delta^{\rho - 1}/|x - y|\}$ values of $e$ in $E_{0}$ such that $x \sim_{e} y$.  Whenever $\delta^{\rho - 1}/|x - y| \geq 1$, this and the inequality $\min\{a,b\} \leq a^{\gamma}b^{1 - \gamma}$ allow us to estimate
\begin{displaymath} \card \{e \in E_{0} : x \sim_{e} y\} \lesssim \min\left\{\frac{\delta^{\rho - 1}}{|x - y|}, M \right\} \leq \frac{\delta^{\gamma(\rho - 1)}}{|x - y|^{\gamma}} \cdot \delta^{-\tau(1 - \gamma)} = \frac{\delta^{\gamma(\rho - 1)- \tau(1 - \gamma)}}{|x - y|^{\gamma}}. \end{displaymath}
Thus,
\begin{align*} \mathcal{E} & \lesssim \iint_{\{|x - y| \geq \delta^{\rho - 1}\}} \, d\mu x \, d\mu y +  \iint_{\{|x - y| \leq \delta^{\rho - 1}\}} \card\{e \in E_{0} : x \sim_{e} y\} \, d\mu x\, d\mu y\\
& \lesssim 1 + \delta^{\gamma(\rho - 1)- \tau(1 - \gamma)} \iint |x - y|^{-\gamma} \, d\mu x \, d\mu y \asymp \max\{1,\delta^{\gamma(\rho - 1)- \tau(1 - \gamma)}\}. \end{align*}
Next, we estimate $\mathcal{E}$ from below in terms of $M$. If $e \in E_{0}$, we have
\begin{displaymath} N(K_{e},\delta^{\rho}) \leq \delta^{-\rho\sigma}, \end{displaymath}
since $\delta^{\rho} \leq \delta < 1/i$. This means that $K$ -- and $\spt \mu$ in particular -- can be covered with some tubes $T_{1},\ldots,T_{K} \in \calT_{e}$ with $K \lesssim \delta^{-\rho\sigma}$. An application of the Cauchy-Schwarz inequality, similar to the one seen in the proof of Theorem \ref{discreteMarstrand}, gives
\begin{align*} \mu \times \mu(\{(x,y) : x \sim_{e} y\}) & = \sum_{j = 1}^{K} \mu \times \mu(\{(x,y) : x,y \in T_{j}\})\\
& = \sum_{j = 1}^{K} \mu(T_{j})^{2} \stackrel{\text{C-S}}{\geq} \frac{1}{K} \left(\sum_{j = 1}^{K} \mu(T_{j}) \right)^{2} \gtrsim \delta^{\rho\sigma}.
\end{align*}
This shows that $\mathcal{E} \gtrsim M \cdot \delta^{\rho\sigma}$, and so
\begin{equation}\label{form34} M \lesssim \delta^{-\rho\sigma} \cdot \max\{1,\delta^{\gamma(\rho - 1) - \tau(1 - \gamma)}\}. \end{equation}
The proof is finished by iterating this estimate. Here is the idea. \textbf{If} 
\begin{equation}\label{form36} \gamma(\rho - 1) - \tau(1 - \gamma) \leq 0, \end{equation} the second term dominates inside the maximum in \eqref{form34}, and we obtain the bound $M \lesssim \delta^{-\rho\sigma + \gamma(\rho - 1) - \tau(1 - \gamma)}$. We may then replace the \emph{a priori} estimate $M \lesssim \delta^{-\tau}$ by $M \lesssim \delta^{-\rho\sigma + \gamma(\rho - 1) - \tau(1 - \gamma)}$ and start the proof over (of course, here we need to know that some \emph{a priori} estimate is \textbf{true} to begin with, but, as noted, we always have $M \lesssim \delta^{-\tau}$ with $\tau = 1$, for example). Continuing in this manner (and assuming that \eqref{form36} always holds), we get a sequence of estimates, where the 'new' exponent of $\delta$ is obtained by multiplying the previous one by $(1 - \gamma) < 1$ and adding $-\rho\sigma + \gamma(\rho - 1)$. After $n \geq 1$ iterations, the result will look like
\begin{displaymath} -\tau_{n} := [-\rho\sigma + \gamma(\rho - 1)]\sum_{k = 0}^{n - 1} (1 - \gamma)^{k} - (1 - \gamma)^{n}\tau. \end{displaymath}
Since $-\tau_{n} \to -\rho\sigma/\gamma + (\rho - 1)$, we see that $M \lesssim \delta^{-\rho\sigma/\gamma + (\rho - 1)}$, and this gives 
\begin{equation}\label{form35} \overline{\dim}_{\B} E \leq \frac{\rho\sigma}{\gamma} - (\rho - 1). \end{equation}
It is immediate from \eqref{form35} that large choices of $\rho$ give better estimates for $\overline{\dim}_{\B} E$. So, how large can we take $\rho$ to be? For the validity of the previous argument, it was crucial that \eqref{form36} was true in every one of the infinite number of iterations: in other words, it seems like we should choose $\rho$ so that \eqref{form36} holds with $\tau$ replaced by $\tau_{n}$, for all $n \in \N$. Fortunately, there is an easier way. Let
\begin{displaymath} \rho := \frac{\gamma}{\gamma + \sigma(\gamma - 1)} \geq 1. \end{displaymath}
Then, there are two alternatives. If \eqref{form36} fails at some iteration (that is, for some $\tau_{n}$) we may read from \eqref{form34} that $M \lesssim \delta^{-\rho\sigma}$. This immediately yields the estimate \eqref{form37}. But if \eqref{form36} holds for every $\tau_{n}$, $n \in \N$, we have \eqref{form35} at our disposal: and with this particular choice of $\rho$, one readily checks that we end up with \eqref{form37} again.
\end{proof}

\begin{proof}[Proof of the second estimate in Theorem \ref{estimate}] The proof begins in a manner similar to the previous one. It suffices to show the following assertion: assume that $\gamma \in (0,1)$, let $K \subset B(0,1)$ be a compact set supporting a Borel probability measure $\mu$ satisfying $\mu(B(x,r)) \lesssim r^{\gamma}$ and $I_{\gamma}(\mu) < \infty$, let $\gamma/2 \leq \sigma < \gamma$, and let $i \in \N$. Then the upper box-dimension of the exceptional set
\begin{displaymath} E := \bigcap_{\delta \in (0,1/i)} \{e \in S^{1} : N(K_{e},\delta) \leq \delta^{-\sigma}\} \end{displaymath}
admits the estimate
\begin{equation}\label{box} \overline{\dim}_{\B} E \leq  \frac{(2\sigma - \gamma)(1 - \gamma)}{\gamma/2} + \sigma. \end{equation}
If $\card E \leq 2$, we are done. Otherwise, choose three distinct vectors $\xi_{1},\xi_{2},\xi_{3} \in E$. We record the following useful property: there exists a constant $\alpha > 0$ such that any vector $e \in S^{1}$ is at distance $\alpha$ from \textbf{at least two} of the vectors $\xi_{1},\xi_{2},\xi_{3}$.

Fix $\delta < 1/i$. Let us recall and redefine some notation from the previous proofs. Given $e \in S^{1}$, we write
\begin{displaymath} \calT_{e} := \{\rho_{e}^{-1}[j\delta,(j + 1)\delta) : j \in \Z\}. \end{displaymath}
Thus, $\calT_{e}$ consists of disjoint half-open $\delta$-tubes, perpendicular to the vector $e$. If $x,y \in \R^{2}$, we define the relation $x \sim_{e} y$, as before, by 
\begin{displaymath} x \sim_{e} y \quad \Longleftrightarrow \quad x,y \in T \in \calT_{e}. \end{displaymath}
Thus, the points $x$ and $y$ have to be contained in the \textbf{same} tube in $\calT_{e}$. Now we define a version of the $\mathcal{E}$-energy. Let $E_{0} \subset E$ be any $\delta$-separated set, and define
\begin{displaymath} \mathcal{E} := \sum_{e \in E_{0}} \iint_{\{(x,y) : x \sim_{e} y\}} |x - y|^{1 - \gamma} \, d\mu x \, d\mu y. \end{displaymath}
Let us first bound $\mathcal{E}$ from above. Again, we make use of the fact that the set $\{e \in S^{1} : x \sim_{e} y\}$ is an arc $J_{x,y}$ of length $\ell(J_{x,y}) \lesssim \delta/|x - y|$. In particular, given a pair of points $x,y \in \R^{2}$, at most $\lesssim |x - y|^{-1}$ vectors $e \in E_{0}$ can satisfy $x \sim_{e} y$. This observation yields
\begin{displaymath} \mathcal{E} = \iint \card \{e \in E_{0} : x \sim_{e} y\} |x - y|^{1 - \gamma} \, d\mu x \, d\mu y \lesssim  \iint |x - y|^{-\gamma} \, d\mu x \, d\mu y \asymp 1. \end{displaymath}

Next, we will bound $\mathcal{E}$ from below in terms of $\card E_{0}$. Fix any vector $e \in E_{0}$. Then $N(K_{e},\delta) \leq \delta^{-\sigma}$, which means that $\spt \mu \subset K$ is covered by some tubes $T_{1},\ldots,T_{k} \in \calT_{e}$ with $k \lesssim \delta^{-\sigma}$. Fix $\tau > 0$, and, for each tube $T_{j}$, choose a $\delta \times \delta^{\tau}$-rectangle $S_{j} \subset T_{j}$, see Figure \ref{fig7}, with the following property. The set $T_{j} \setminus S_{j}$ has two $\delta^{\tau}$-separated components, say $T_{j}^{-}$ and $T_{j}^{+}$. We choose the position of the rectangle $S_{j}$ so that either
\begin{equation}\label{form16} \mu(T_{j} \setminus S_{j}) \leq c\delta^{\sigma} \quad \text{ or } \quad \mu(T_{j}^{-}) \asymp \mu(T_{j}^{+}),  \end{equation}
where $c > 0$ is a constant so small that $k \cdot c\delta^{\sigma} \leq 1/4$. This means that \textbf{if} we can choose the rectangle $S_{j}$ so that the first option in \eqref{form16} holds, then we do just that. But if no such choice of $S_{j}$ is possible, then, for any choice of $S_{j}$, the opposite must hold: $\mu(T_{j}^{-}) + \mu(T_{j}^{+}) = \mu(T_{j} \setminus S_{j}) > c\delta^{\sigma}$. Now, if we move $S_{j}$ by an amount of $\delta$ up or down the tube $T_{j}$, the $\mu$-measures of the half-tubes $T_{j}^{-}$ and $T_{j}^{+}$ can change by no more than $\lesssim \delta^{\gamma}$, which is much smaller than $c\delta^{\sigma}$ for small values of $\delta$. This ensures that the second option in \eqref{form16} can be attained for a suitable choice of the position of $S_{j}$ (at least if $\delta$ is small enough, which we can always assume).
\begin{figure}[h!]
\begin{center}
\includegraphics[scale = 0.8]{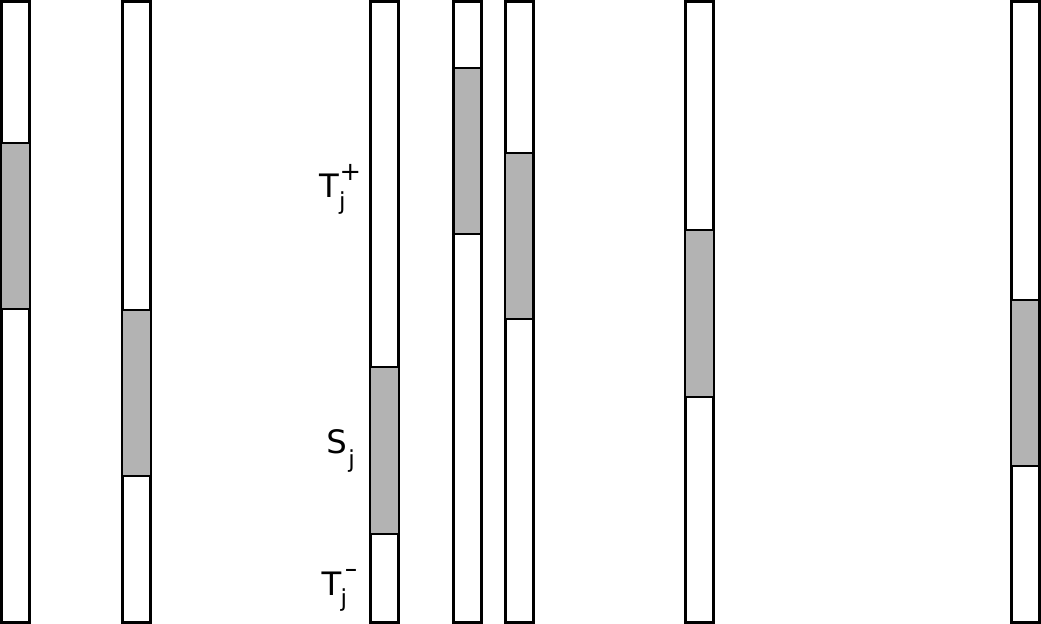}
\caption{The tubes $T_{j}$ and the rectangles $S_{j}$.}\label{fig7}
\end{center}
\end{figure}

Next, we claim that
\begin{equation}\label{form17} \sum_{j = 1}^{k} \mu(T_{j} \setminus S_{j}) \geq \frac{1}{2}, \end{equation} 
for \textbf{large} enough $\tau > 0$ (equivalently, for small enough $\delta^{\tau}$). To prove this, assume that \eqref{form17} fails. Since the total $\mu$-mass of the tubes $T_{j}$ is one, this implies that
\begin{equation}\label{form19} \sum_{j = 1}^{k} \mu(S_{j}) \geq \frac{1}{2}. \end{equation}
We will now use \eqref{form19} to extract a lower bound for $\delta^{\tau}$. We may and will further assume that every rectangle $S_{j}$ has $\mu$-measure at least $c\delta^{\sigma}$: if this is not true to begin with, simply discard all the rectangles with $\mu(S_{j}) < c\delta^{\sigma}$ to obtain a subcollection of some \emph{remaining} rectangles $S_{j}$, which satisfy $\mu(S_{j}) \geq c\delta^{\sigma}$. Then \eqref{form19} holds with $1/2$ replaced by $1/4$, since the total $\mu$-measure of the discarded rectangles $S_{j}$ is bounded by $k \cdot c\delta^{\sigma} \leq 1/4$. We keep the same notation for these remaining rectangles.

It is time to recall the vectors $\xi_{1},\xi_{2},\xi_{3} \in E$ that were chosen at the beginning of the proof. As we remarked upon choosing these $\xi_{j}$, we may find two among the three vectors, say $\xi_{1}$ and $\xi_{2}$, such that $|e - \xi_{1}| \geq \alpha$ and $|e - \xi_{2}| \geq \alpha$. We use this information as follows:
\begin{claim}\label{claim5} Let $P \subset \R^{2}$ be any set, which is contained in a single $\delta$-tube $T \in \calT_{e}$. Then
\begin{displaymath} N(P_{\xi_{j}},\delta) \gtrsim N(P,\delta), \qquad j \in \{1,2\}, \end{displaymath}
where the implicit constants depend only on $\alpha$.
\end{claim}
\begin{proof} If $x,y \in P$ and $|x - y| \geq C\delta$, then the line segment $l$ connecting $x$ and $y$ is almost perpendicular to $e$. In particular, for large enough $C > 0$, we have that $l$ cannot be perpendicular to $\xi_{j}$, and this gives $|\rho_{\xi_{j}}(x) - \rho_{\xi_{j}}(y)| \gtrsim \delta$. 
\end{proof}
We apply the claim with $P_{j} := \spt \mu \cap S_{j}$, for each of the remaining rectangles $S_{j}$. Note that since $\mu(S_{j}) \geq c\delta^{\sigma}$, and $\mu$ satisfies the power bound $\mu(B(x,\delta)) \lesssim \delta^{\gamma}$, we have $N(P_{j},\delta) \gtrsim \delta^{\sigma - \gamma}$. Similarly, it follows from the condition $\sum \mu(S_{j}) \geq 1/4$ that
\begin{equation}\label{form21} N\left(\bigcup_{j = 1}^{k} P_{j},\delta^{\tau}\right) \gtrsim \delta^{-\gamma\tau}. \end{equation}
Since the vectors $\xi_{1}$ and $\xi_{2}$ are $\alpha$-separated (which means that they are essentially orthogonal), we may deduce that either
\begin{equation}\label{form20} N\left(\bigcup_{j = 1}^{k} \rho_{\xi_{1}}(P_{j}),\delta^{\tau} \right) \gtrsim \delta^{-\gamma\tau/2} \quad \text{or} \quad N\left(\bigcup_{j = 1}^{k} \rho_{\xi_{2}}(P_{j}),\delta^{\tau} \right) \gtrsim \delta^{-\gamma\tau/2}, \end{equation}
where the implicit constants depend only on $\alpha$ and the implicit constants in \eqref{form21}. Namely, if both inequalities failed, we could easily cover $\bigcup P_{j}$ with $\ll \delta^{-\gamma\tau}$ balls of radius $\delta^{\tau}$, contradicting \eqref{form21}. Suppose, for example, the the first inequality in \eqref{form20} holds. Then we may choose a $5\delta^{\tau}$-separated subset
\begin{displaymath} R \subset \bigcup_{j = 1}^{k} \rho_{\xi_{1}}(P_{j}) \end{displaymath}
of cardinality $\card R \gtrsim \delta^{-\gamma\tau/2}$. For each point $t \in R$, we may find an index $j(t) \in \{1,\ldots,k\}$ such that $t \in \rho_{\xi_{1}}(P_{j(t)})$. But since $P_{j(t)} \subset S_{j(t)}$, we see that $\rho_{\xi_{1}}(P_{j(t)}) \subset [t - 2\delta^{\tau},t + 2\delta^{\tau}]$. This means that the projections $\rho_{\xi_{1}}(P_{j(t)})$ are $\delta^{\tau}$-separated for distinct $t \in R$. Now, it remains to use Claim \ref{claim5} to deduce the lower bound  
\begin{displaymath} N(\rho_{\xi_{1}}(P_{j}),\delta) \gtrsim N(P_{j},\delta) \gtrsim \delta^{\sigma - \gamma} \end{displaymath}
for every $j \in \{1,\ldots,k\}$, and, in particular, for every $j = j(t)$. It follows that 
\begin{displaymath} N\left(\bigcup_{j = 1}^{k} \rho_{\xi_{1}}(P_{j}),\delta\right) \geq \sum_{t \in R} N(\rho_{\xi_{1}}(P_{j(t)}),\delta) \gtrsim \delta^{-\gamma\tau/2} \cdot \delta^{\sigma - \gamma}. \end{displaymath}
On the other hand, we have $\xi_{1} \in E$, which means that
\begin{displaymath} N\left(\bigcup_{j = 1}^{k} \rho_{\xi_{1}}(P_{j}),\delta\right) \leq N(K_{\xi_{1}},\delta) \leq \delta^{-\sigma}. \end{displaymath}
Comparing the estimates leads to the existence of a constant $b > 0$, independent of $\delta$, such that $\delta^{\tau} > b\delta^{(2\sigma - \gamma)/(\gamma/2)}$. All this was deduced solely on the basis of \eqref{form17} failing. Thus, if
\begin{equation}\label{form22} \delta^{\tau} = b\delta^{(2\sigma - \gamma)/(\gamma/2)}, \end{equation}
we see that \eqref{form17} must hold.

Now we are prepared to estimate $\mathcal{E}$ from below. Choose $\tau > 0$ in such a manner that \eqref{form17} holds. As we just demonstrated, the choice giving $\delta^{\tau} = b\delta^{(2\sigma - \gamma)/(\gamma/2)}$ is ok. Since \eqref{form17} holds, we may discard the indices $j \in \{1,\ldots,k\}$ such that the first possibility in \eqref{form16} is realized: for the remaining indices $j$, say $j \in \{1,\ldots,K\}$, $K \leq k \lesssim \delta^{-\sigma}$, the latter option in \eqref{form16} holds, and, moreover, we still have
\begin{equation}\label{form23} \sum_{j = 1}^{K} \mu(T_{j}\setminus S_{j}) \geq \frac{1}{4} \end{equation}
by the choice of $c$. Here is the reason why we are so interested in removing a (large) rectangle $S_{j}$ from $T_{j}$: if $x \in T_{j}^{-}$ and $y \in T_{j}^{+}$, we have $|x - y| \geq \delta^{\tau}$. This means that we can make the following estimate:
\begin{align*} \iint_{\{(x,y) : x\sim_{e} y\}} |x - y|^{1 - \gamma} \, d\mu x \, d\mu y & \geq \sum_{j = 1}^{K} \int_{x \in T_{j}^{-}} \int_{y \in T_{j}^{+}} |x - y|^{1 - \gamma} \, d\mu x \, d\mu y\\
& \geq \delta^{\tau(1 - \gamma)} \cdot \sum_{j = 1}^{K} \mu(T_{j}^{-}) \cdot \mu(T_{j}^{+})\\
& \stackrel{\eqref{form16}}{\asymp} \delta^{\tau(1 - \gamma)} \cdot \sum_{j = 1}^{K} \mu(T_{j}^{+})^{2}\\
& \stackrel{\text{C-S}}{\geq} \delta^{\tau(1 - \gamma)} \cdot \frac{1}{K} \left(\sum_{j = 1}^{K} \mu(T_{j}^{+}) \right)^{2}\\
&  \stackrel{\eqref{form16}}{\gtrsim} \delta^{\tau(1 - \gamma) + \sigma} \left(\sum_{j = 1}^{K} \mu(T_{j} \setminus S_{j}) \right)^{2} \stackrel{\eqref{form23}}{\gtrsim} \delta^{\tau(1 - \gamma) + \sigma}. \end{align*}
The letters C-S refer to Cauchy-Schwarz. This estimate holds uniformly for every vector $e \in E_{0}$, so we have
\begin{displaymath} \delta^{\tau(1 - \gamma) + \sigma} \cdot \card E_{0} \lesssim \mathcal{E} \lesssim 1. \end{displaymath}
This yields
\begin{displaymath} N(E,\delta) \lesssim \delta^{-\sigma - \tau(1 - \gamma)}\end{displaymath} 
\textbf{for any such $\tau > 0$ such that \eqref{form17} holds}. The choice of $\tau$ indicated by \eqref{form22} immediately yields the bound \eqref{box}.
\end{proof}

Next, we use a similar method to prove Theorem \ref{estimate2}. The idea is this: the last few lines of the previous proof reveal that if we could always choose $\tau$ arbitrarily close to zero, we would immediately obtain $\overline{\dim}_{\B} E \leq \sigma$. The problem with general sets is that such a choice might result in the failure of the crucial estimate \eqref{form17}: this would essentially mean that, simultaneously, the dimension of the projection in some direction $e \in E$ drops to $\sigma < \gamma$ \textbf{and} most of the measure $\mu$ is concentrated in the $\delta^{\tau}$-neighbourhood of a graph 'above' the line spanned by the vector $e$. For self-similar sets and measures (under some additional conditions, at least), such behavior is simply not possible for $\tau > 0$. The reason for this is the following \emph{dimension conservation principle} introduced by H. Furstenberg.

\begin{definition}[Dimension conservation principle] Let $K \subset \R^{2}$. A projection $\rho_{e} \colon \R^{2} \to \R$ is \emph{dimension conserving}, if there exists $\Delta = \Delta(e) \geq 0$ such that
\begin{displaymath} \Delta + \dim \{t \in \R : \dim [K \cap \rho_{e}^{-1}\{t\}] \geq \Delta\} \geq \dim K. \end{displaymath}
In this definition, the convention is adopted that $\dim \emptyset = -\infty$: this means, among other things, that $\Delta = \dim K$ is an admissible choice for $\Delta$ only in case there exist some lines $\rho_{e}^{-1}\{t\}$ such that $\dim [K \cap \rho_{e}^{-1}\{t\}] = \dim K$. Also, if $\rho_{e}^{-1}\{t\} \cap K = \emptyset$, we have $\dim [K \cap \rho_{e}^{-1}\{t\}] = -\infty$, which means that 
\begin{equation}\label{form24} \{t : \dim [K \cap \rho_{e}^{-1}\{t\}] \geq \delta\} \subset K_{e}. \end{equation}
\end{definition}

\begin{remark} There is no reason why $\Delta(e)$ should be unique, so, in fact, the notation $\Delta(e)$ refers to a set. Whenever we write $\Delta(e) \geq C$, we mean that 
\begin{displaymath} \sup \Delta(e) \geq C. \end{displaymath}
\end{remark}
The requirement $\inf \Delta(e) \geq C$ might seem more natural, but this definition makes Proposition \ref{prop2} slightly stronger. In \cite[Theorem 6.2]{Fu} Furstenberg proves that if $K \subset \R^{2}$ is a \emph{compact homogeneous set}, then every projection $\rho_{e}$, $e \in S^{1}$, is dimension conserving. For the precise definition of homogeneous sets, we refer to \cite[Definition 1.4]{Fu}, but for Theorem \ref{estimate2} in mind, it suffices to know two facts: (i) all self-similar sets in the plane containing no rotations and satisfying the strong separation condition are homogeneous, and (ii) all compact homogeneous sets $K$ have $\dim K = \overline{\dim}_{\B} K$. Both facts are stated immediately after \cite[Definition 1.7]{Fu}. We will use Furstenberg's result via the following easy proposition:
\begin{proposition}\label{prop1} Let $K \subset \R^{2}$ be a compact homogeneous set. Then
\begin{displaymath} \{e \in S^{1} : \dim K_{e} \leq \sigma\} \subset \{e \in S^{1} : \Delta(e) \geq \dim K - \sigma\}. \end{displaymath}
\end{proposition}

\begin{proof} According to Furstenberg's result, we know that every projection $\rho_{e}$ is dimension conserving, so that $\Delta(e)$ is well-defined. Suppose that $\dim K_{e} \leq \sigma$. If, in the set $\Delta(e)$, there was even one value $\Delta$ with $\Delta < \dim K - \sigma$, we would immediately obtain
\begin{displaymath} \dim K \leq \Delta + \dim \{t : \dim [K \cap \rho_{e}^{-1}\{t\}] \geq \Delta\} \stackrel{\eqref{form24}}{\leq} \Delta + \dim K_{e} <  \dim K, \end{displaymath}
which is absurd. Hence, $\dim K_{e} \leq \sigma$ even implies $\inf \Delta(e) \geq \dim K - \sigma$.
\end{proof}

Thus, for compact homogeneous sets, we may estimate the packing dimension of the exceptional set $\{e \in S^{1} : \Delta(e) \geq \dim K - \sigma\}$ instead of $\{e \in S^{1} : \dim K_{e} \leq \sigma\}$. Such an estimate is the content of the following proposition.

\begin{proposition}\label{prop2} Let $K \subset \R^{2}$ be a compact set with $\dim K = \overline{\dim}_{\B} K = \gamma$, and let $0 \leq \sigma < \gamma$. Then $\dim_{\p} E \leq \sigma$, where
\begin{displaymath} E = \{e \in S^{1} : \rho_{e} \text{ is dimension conserving, and } \Delta(e) \geq \gamma - \sigma\}. \end{displaymath}
\end{proposition}

\begin{proof} If the projection $\rho_{e}$ is dimension conserving, and $\Delta \in \Delta(e)$, then for any $\tau > 0$ we may find $\varepsilon > 0$ such that
\begin{displaymath} H^{\gamma - \Delta - \tau}(\{t : H^{\Delta - \tau}(K \cap \rho_{e}^{-1}\{t\}) > \varepsilon\}) > \varepsilon, \end{displaymath}
where $H^{d}$ stands for $d$-dimensional Hausdorff content. This reduces us to proving the estimate 
\begin{equation}\label{form25} \overline{\dim}_{\B} E_{\varepsilon,\tau} \leq \sigma + 3\tau. \end{equation}
for any $\varepsilon > 0$ and $0 < \tau < \gamma - \sigma$, where 
\begin{displaymath} E_{\varepsilon,\tau} := \{e \in S^{1} : H^{\gamma - \Delta - \tau}(\{t : H^{\Delta - \tau}(K \cap \rho_{e}^{-1}\{t\}) > \varepsilon\}) > \varepsilon \text{ for some } \Delta \geq \gamma - \sigma\}. \end{displaymath}
 Fix $\delta > 0$. At this point, we should mention that in the $\asymp$ and $\lesssim$ notation below, all implicit constants may depend on $\varepsilon,\gamma,K,\sigma$ and $\tau$, but \textbf{not} on $\delta$. Since $\overline{\dim}_{\B} K = \gamma$, we may choose a collection of points $K_{0} := \{x_{1},\ldots,x_{N}\} \subset K$ such that $N \asymp \delta^{-\gamma}$, and
\begin{displaymath} K \subset \bigcup_{n = 1}^{N} B(x_{n},\delta). \end{displaymath}
Given $e \in S^{1}$, define the $\delta$-tubes $\calT_{e}$ by
\begin{displaymath} \calT_{e} = \{\rho_{e}^{-1}[j\delta,(j + 1)\delta) : j \in \Z\}. \end{displaymath}
Let $d = (\gamma - \sigma - \tau)^{-1}$. We define the relation $x \sim_{e} y$ for $x,y \in \R^{2}$:
\begin{displaymath} x \sim_{e} y \quad \Longleftrightarrow \quad |x - y| \geq \left(\frac{\varepsilon}{10}\right)^{d} \text{ and } B(x,\delta) \cap T \neq \emptyset \neq B(y,\delta) \cap T \text{ for some } T \in \calT_{e}. \end{displaymath}
This definition differs from its analogues in the previous proofs in that now we require the points $x$ and $y$ to be separated by a constant independent of $\delta$, and also the strict inclusion $x,y \in T$ is relaxed to $x$ and $y$ being relatively close to a single tube in $\calT_{e}$. Let $E_{0} \subset E_{\varepsilon,\tau}$ be any $\delta$-separated finite set. The energy $\mathcal{E}$ is defined as follows:
\begin{displaymath} \mathcal{E} := \sum_{e \in E_{0}} \card \{(x,y) \in K_{0} \times K_{0} : x \sim_{e} y\}. \end{displaymath}
Once more, we intend to estimate $\mathcal{E}$ from above and below. The estimate from above is easy. If $x,y \in K_{0}$, the number of vectors $e \in E_{0}$ such that $x \sim_{e} y$ is bounded by a constant depending only on $\varepsilon,\gamma,\sigma$ and $\tau$ -- but not on $\delta$. Hence, $\mathcal{E} \lesssim N^{2} \asymp \delta^{-2\gamma}$. To bound $\mathcal{E}$ from below, fix $e \in E_{0}$. By definition of $E_{\varepsilon,\tau}$, there exist $\Delta \geq \gamma - \sigma$ and tubes $T_{1},\ldots,T_{k} \in \calT_{e}$ such that $k \gtrsim \delta^{\Delta + \tau - \gamma}$, and every tube $T_{j}$ contains a line $L_{j} := \rho_{e}^{-1}\{t_{j}\}$ with
\begin{displaymath} H^{\Delta - \tau}(K \cap L_{j}) > \varepsilon. \end{displaymath}
Consider a fixed tube $T_{j}$. If $\delta < (\varepsilon/9)^{d}$, then, by the choice of $d$, the $(\Delta - \tau)$-dimensional Hausdorff content of a rectangle $S$ with dimensions $\delta \times (\varepsilon/9)^{d}$ is no more than $\varepsilon/2$. This implies that 
\begin{equation}\label{form26} H^{\Delta - \tau}([K \cap L_{j}] \setminus S) > \varepsilon/2 \end{equation}
for any such rectangle $S$. A $\delta$-cover of the set $[K \cap L_{j}] \setminus S$ is obtained by all the balls $B(x_{n},\delta)$, $x_{n} \in K_{0}$, which have non-empty intersection with $[K \cap L_{j}] \setminus S$. According to \eqref{form26}, there must be $\gtrsim \delta^{\tau - \Delta}$ such balls, for any choice of $S$. Now, as in the previous proof, we simply choose $S \subset T_{j}$ in such a manner that $T_{j} \setminus S$ is divided into two disjoint $(\varepsilon/9)^{d}$-separated half-tubes $T_{j}^{+}$ and $T_{j}^{-}$ so that
\begin{displaymath} \card \{x_{n} : B(x_{n},\delta) \cap [K \cap L_{j} \cap T_{j}^{\pm}] \neq \emptyset\} \gtrsim \delta^{\tau - \Delta}. \end{displaymath}
Finally, if $x_{m},x_{n} \in K_{0}$ are points such that $B(x_{m},\delta) \cap [K \cap L_{j} \cap T_{j}^{-}] \neq \emptyset$ and $B(x_{n},\delta) \cap [K \cap L_{j} \cap T_{j}^{+}] \neq \emptyset$, we have $|x_{n} - x_{m}| \geq (\varepsilon/9)^{d} - 2\delta \geq (\varepsilon/10)^{d}$ for small enough $\delta$, and this shows that $x_{m} \sim_{e} x_{n}$. By the choice of $S$, there are $\gtrsim \delta^{2(\tau - \Delta)}$ pairs $(x_{m},x_{n})$ with this property. Now we would like to make the estimate
\begin{align*} \card\{(x,y) \in K_{0} \times K_{0} : x \sim_{e} y\} & \gtrsim k \cdot \delta^{2(\tau - \Delta)}\\
& \gtrsim \delta^{\Delta + \tau - \gamma + 2(\tau - \Delta)}\\
& = \delta^{3\tau - \Delta - \gamma} \geq \delta^{3\tau + \sigma - 2\gamma}, \end{align*}
the last inequality being equivalent with $\Delta \geq \gamma - \sigma$. This is correct, but one must be a bit careful, since, in the first inequality, any pair of points $(x_{m},x_{n})$ may be counted several times, if $B(x_{m},\delta) \cap [K \cap L_{j} \cap T_{j}^{-}] \neq \emptyset$ and $B(x_{n},\delta) \cap [K \cap L_{j} \cap T_{j}^{+}] \neq \emptyset$ for multiple indices $j$. We are saved by the fact that any ball of radius $\delta$ may intersect no more than three tubes $T_{j}$, so each pair $(x_{m},x_{n})$ gets counted no more than nine times. This implies that $\mathcal{E}$ can be bounded from below as
\begin{displaymath} \mathcal{E} \gtrsim \card E_{0} \cdot \delta^{3\tau + \sigma - 2\gamma}, \end{displaymath}
and so we have proved that
\begin{displaymath} \card E_{0} \lesssim \delta^{- \sigma - 3\tau}. \end{displaymath}
This gives \eqref{form25} and concludes the proof of the proposition.
\end{proof}

We will now finish the proof of Theorem \ref{estimate2}.

\begin{proof}[Proof of Theorem \ref{estimate2}] If $K$ is compact and homogeneous, it follows from \cite{Fu} that $\dim K = \overline{\dim}_{\B} K$. Thus, the part of Theorem \ref{estimate2} for compact homogeneous sets follows immediately by combining Propositions \ref{prop1} and \ref{prop2}. 

Next, let $K \subset \R^{2}$ be a self-similar set with $\dim K = \gamma$, and let $0 \leq \sigma < \gamma$. If $K$ contains an irrational rotation, it follows from \cite[Theorem 5]{PS} that $\dim K_{e} = \gamma$ for every direction $e \in S^{1}$. So, we may assume that $K$ contains no irrational rotations. Then \cite[Lemma 4.2]{Or} shows that there exists a self-similar set $\tilde{K} \subset K$ satisfying the strong separation condition, containing no rotations, and with $\tilde{\gamma} = \dim \tilde{K} > \sigma$. According to \cite{Fu}, the set $\tilde{K}$ is homogeneous, and certainly also $\overline{\dim}_{\B} \tilde{K} = \tilde{\gamma}$. Hence, it follows from Propositions \ref{prop1} and \ref{prop2} that the set
\begin{displaymath} \tilde{E} := \{e \in S^{1} : \dim \tilde{K}_{e} \leq \sigma\}. \end{displaymath}
satisfies $\dim_{\p} \tilde{E} \leq \sigma$. The proof is finished by observing that 
\begin{displaymath} \{e \in S^{1} : \dim K_{e} \leq \sigma\} \subset \tilde{E}. \end{displaymath}
\end{proof}

\section{The example in Theorem \ref{bigEx}}

Let us say a few words to explain our motivation to see through the construction presented below. If $K \subset \R^{2}$ is a self-similar fractal containing no rotations, then $\dim K_{e} = \dim_{\p} K_{e} = \overline{\dim}_{\B} K_{e}$ for every vector $e \in S^{1}$. It is a long-standing problem, attributed to H. Furstenberg, see \cite[Question 2.5]{PSo}, to determine the largest possible size of the exceptional set $\{e \in S^{1} : \dim K_{e} < \dim K\}$, given that $K \subset \R^{2}$ is self-similar without rotations and $\dim K \leq 1$. It is conjectured that this set should be no more than countable. At some point, it occurred to us that perhaps this conjecture could be verified by showing that the set $\{e \in S^{1} : \dim_{\p} K_{e} < \dim K\}$ is always at most countable, for \textbf{any} set Borel set $K \subset \R^{2}$ with $\dim K \leq 1$. These dreams were put to rest by the emergence of the construction below. The seemingly stronger conclusion in Theorem \ref{bigEx} that the exceptional set may even have large packing dimension is practically free of charge: the construction would be no less tedious, were we only interested in the uncountability of the set $\{e \in S^{1} : \dim_{\p} K_{e} < \dim K\}$. Finally, it is still possible that the approach via general sets and the packing dimension of projections could be used to prove a weaker form of Furstenberg's conjecture, namely that $\dim \{e \in S^{1} : \dim K_{e} < \dim K\} = 0$ for self-similar sets $K \subset \R^{2}$ as above.

Another point worth mentioning relates our example to a 'number theoretic' construction from the 70's. In \cite{KM}, Kaufman and Mattila prove that Kaufman's bound \eqref{kaufman} is sharp by presenting a Borel set $B \subset \R^{2}$ of Hausdorff dimension $\dim B = s \in (0,1]$ such that $\dim \{e \in S^{1} : \dim B_{e} < \dim B\} = s$. It is fair to ask, whether, by lucky coincidence, the projections of the set $B$ might also have small packing dimension: this could potentially be a major trouble-saver and an improvement to Theorem \ref{bigEx}! There is a simple reason why this idea fails: the example of Kaufman and Mattila is a set $B$ of the second category in the plane. Every continuous open surjection, including projections, take sets of second category to sets of second category. It follows immediately that $\dim_{\p} B_{e} = 1$ for every $e \in S^{1}$.

\subsection{Proof of Theorem \ref{bigEx}} We begin by setting up some notation. Let $K_{1}, K_{2} \subset B(0,1/2)$ be compact sets, which are expressible as the unions of certain finite collections $\calG_{1}$ and $\calG_{2}$ of closed balls with disjoint interiors. We define a new set $K_{1} \star K_{2} \subset B(0,1/2)$ by 'taking all the balls in $\calG_{2}$ and scaling and translating them inside each and every ball in $\calG_{1}$'. Formally, if $B \subset \R^{2}$ is a closed ball, let $T_{B}$ be the linear transformation taking $B(0,1/2)$ to $B$ without rotations. Then
\begin{equation}\label{formalDef} K_{1} \star K_{2} := \bigcup_{B \in \calG_{1}} T_{B}(K_{2}). \end{equation}
The set $K_{1} \star K_{2} \subset B(0,1/2)$ is again compact and expressible as the union of $[\card \calG_{1}]\cdot[\card \calG_{2}]$ closed balls with disjoint interiors. The abbreviation 
\begin{displaymath} K^{(m)} := K \star K \star \cdots \star K \end{displaymath}
will be used to denote the $m$-fold $\star$-product of a set $K \subset B(0,1/2)$ with itself. Finally, if $K \subset B(0,1/2)$ is a set expressible as the union of finitely many balls with disjoint interiors, the centers of these balls form a finite set $S_{K} \subset K$, the \emph{skeleton} of $K$. We record some useful relations between $\star$-products and orthogonal projections.

\begin{lemma}\label{projCover} Let $K_{1}, K_{2} \subset B(0,1/2)$ be sets expressible as the finite unions of balls with disjoint interiors, and let $e \in S^{1}$. Then
\begin{displaymath} \card \rho_{e}(S_{K_{1} \star K_{2}}) \leq [\card \rho_{e}(S_{K_{1}})] \cdot [\card \rho_{e}(S_{K_{2}})]. \end{displaymath}
Assume, furthermore, that all the $K_{1}$-balls have common diameter $\delta_{1} \in (0,1]$. Then
\begin{displaymath} N(\rho_{e}(K_{1} \star K_{2}),\delta) \leq N(\rho_{e}(K_{1}),\delta), \qquad \delta > 0, \end{displaymath}
and
\begin{displaymath} N(\rho_{e}(K_{1} \star K_{2}),\delta) \leq [\card \rho_{e}(S_{K_{1}})] \cdot N\left(\rho_{e}(K_{2}),\frac{\delta}{\delta_{1}}\right), \quad \delta > 0. \end{displaymath} 
\end{lemma}

\begin{proof} The first inequality is clear and the second follows from $K_{1} \star K_{2} \subset K_{1}$. To prove the remaining inequality, fix $\delta > 0$. Write $\calG_{1}$ for the collection of balls, the union of which is $K_{1}$. Observe that 
\begin{displaymath} N(\rho_{e}[T_{B}(K_{2})],\delta) = N\left(\rho_{e}(K_{2}),\frac{\delta}{\delta_{1}}\right), \qquad B \in \calG_{1}. \end{displaymath}
If $B_{1} = B(x,\delta_{1}) \in \calG_{1}$ and $B_{2} = B(x_{2},\delta_{1}) \in \calG_{1}$ are balls such that $\rho_{e}(x_{1}) = \rho_{e}(x_{2})$, then also $\rho_{e}[T_{B_{1}}(K_{2})] = \rho_{e}[T_{B_{2}}(K_{2})]$. Now the desired estimate follows from \eqref{formalDef}.
\end{proof}

Next, we will introduce, for each $n \in \N$, a compact set $B_{n} \subset B(0,1/2)$, which is expressible as the union of a large but finite collection of closed balls with disjoint interiors and a common diameter. These sets will play the role of 'basic building blocks' in our construction. Indeed, the desired set $K$ will be defined by
\begin{equation}\label{form15} K = \lim_{j \to \infty} (( \cdots ((B_{n_{1},e_{1}}^{(m_{1})} \star B_{n_{2},e_{2}})^{(m_{2})} \star B_{n_{3},e_{3}})^{(m_{3})} \star \cdots )^{(m_{j - 1})} \star B_{n_{j},e_{j}})^{(m_{j})}. \end{equation}
where $B_{n,e}$ refers to a rotated copy of $B_{n}$.
\begin{figure}[h!]
\begin{center}
\includegraphics[scale = 0.8]{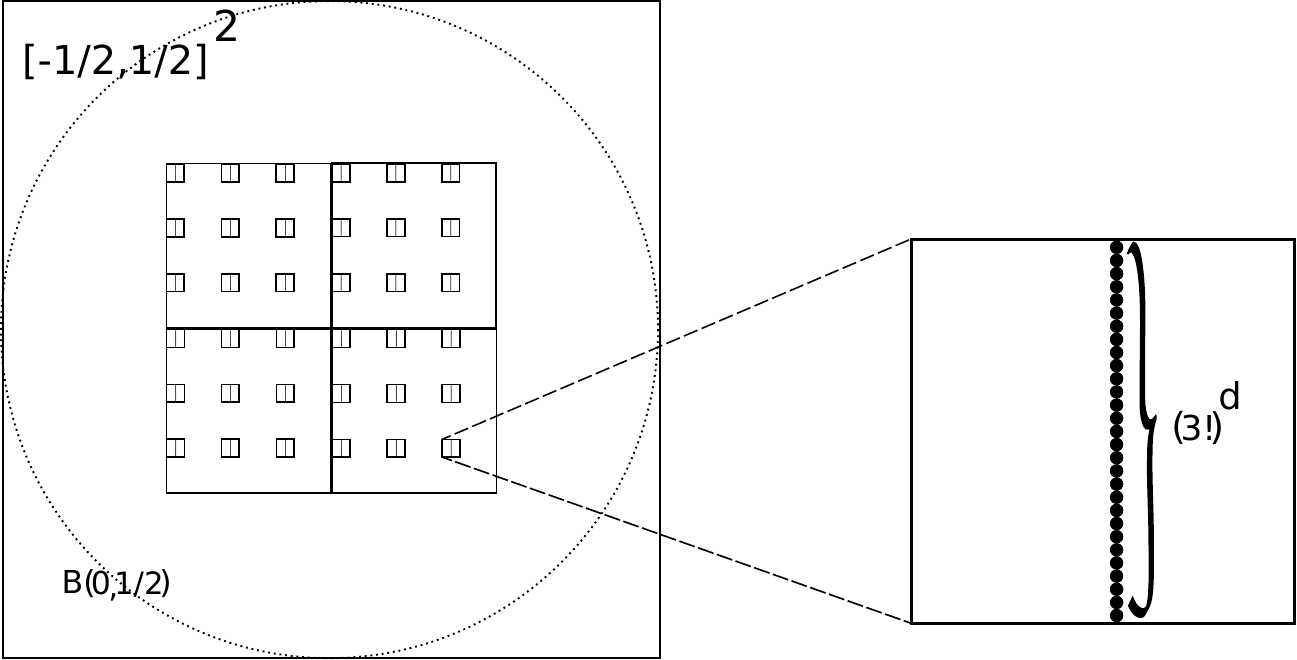}
\caption{The sets $U_{3}$ and $B_{3}$.}\label{U3}
\end{center}
\end{figure}

The set $K_{3}$ is depicted in Figure \ref{U3}. To define $B_{n}$ for general $n$, it is handy to use a variant of the $\star$-product for square collections. If $Q \subset \R^{2}$ is a closed square, let $T_{Q}$ be the linear transformation taking the unit square $[-1/2,1/2]^{2}$ onto $Q$ without rotations. If $K_{1}, K_{2} \subset [-1/2,1/2]^{2}$ are compact sets expressible as the finite unions certain collections $\calG_{1}$ and $\calG_{2}$ of closed suqares with disjoint interiors, define $K_{1} \star K_{2}$ by the familiar formula \eqref{formalDef}, just replacing the two occurences of $B$ by $Q$. Then, in order to define $B_{n}$, 
\begin{itemize}
\item[(a)] let $Q_{1} = [-1/2,1/2]^{2}$, and let $Q_{2} \subset B(0,1/2) \subset [-1/2,1/2]^{2}$ be the set consisting of the four closed squares of side-length $1/4$ and disjoint interiors, which all have a common corner at $(0,0)$,
\item[(b)] let $Q_{n} \subset [-1/2,1/2]^{2}$, $n \geq 3$, be the set consisting of $n^{2}$ closed squares of side-length $n^{-2}$ placed inside the unit square in such a manner that the midpoints form a grid homothetic to $\{1,\ldots,n\}^{2}$, and the distance between vertically or horizontally neighboring midpoints is $n^{-1}$. To specify $Q_{n}$ uniquely, we agree that the to left square has a common corner with $[-1/2,1/2]$. 
\item[(c)] Fix $d \geq 3$, and let $L_{n} \subset [-1/2,1/2]^{2}$, $n \geq 1$, be the set consisting of $(n!)^{d}$ closed squares of side-length $(n!)^{-d}$ and disjoint interiors, whose midpoints lie on the $y$-axis.
\end{itemize}

We write
\begin{displaymath} U_{n} := Q_{1} \star Q_{2} \star \cdots \star Q_{n}, \qquad n \geq 1. \end{displaymath}
The set $U_{3}$ is visible in Figure \ref{U3}. The set $B_{n}$ is defined by replacing every one of the $(n!)^{2 + d}$ squares of $U_{n} \star L_{n}$ by a concentric ball of radius $(n!)^{-2 - d}$. The set $B_{3}$ is also visible in Figure \ref{U3}. The only reason why we had to define $Q_{1}$ differently from the other sets $Q_{n}$ was to ensure that $B_{n} \subset B(0,1/2)$ for all $n \in \N$. For convenience, we also define $B_{0} := (0,1/2)$.

Recalling Lemma \ref{grid}, we say that a direction $e \in S^{1}$ is \emph{rational}, if $e = c(1,pq^{-1})$ for some integers $p,q \in \Z$, $q \neq 0$, and $c = (1 + p^{2}q^{2})^{-1/2}$. The definition of the sets $U_{n}$ and $B_{n}$ may seem complicated, but the precise structure is only needed in the proof of the following lemma; for the rest of the proof of Theorem \ref{bigEx}, we can simply refer to the three properties stated below.

\begin{lemma}\label{heart} Let $e = c(1,pq^{-1}) \in S^{1}$ be a rational direction, let $1/2 < s < 1$, and let $(1 + d)/(2 + d) < t < 1$. Then
\begin{itemize}
\item[(i)] There exists $\delta_{e,s} > 0$ such that 
\begin{displaymath} N(\rho_{e}(B_{n}),\delta) \leq \delta^{-s}, \qquad (n!)^{-2} \leq \delta \leq \delta_{e,s},\: n \in \N. \end{displaymath}
Note that if $(n!)^{-2} > \delta_{e,s}$, the claim says nothing. Moreover,
\begin{displaymath} N(\rho_{e}(B_{n}),\delta) \lesssim_{e,t} \delta^{-t}, \qquad (n!)^{-2 - d} \leq \delta \leq 1. \end{displaymath}
\item[(ii)] Let $S_{n}$ be the skeleton of the set $B_{n}$, that is, $S_{n} = S_{B_{n}}$. Then there exists $n_{e} \in \N$ such that
\begin{displaymath} \card \rho_{e}(S_{n}) \leq (n!)^{t(2 + d)}, \qquad n \geq n_{e}. \end{displaymath}
\item[(iii)] Let $e \in S^{1}$ be a rational direction such that the lines $L = \rho_{e}^{-1}\{t\}$ have negative slope $k(n!)^{-d}$ for some $k \in \{1,\ldots,(n!)^{d - 3}\}$. This simply means that $L$ can be written in the form
\begin{displaymath} L = \{(x,y) : y = -k(n!)^{-d}x + y_{0}\}, \qquad 1 \leq k \leq (n!)^{d - 3}. \end{displaymath} 
The collection of these $(n!)^{d - 3}$ directions will be denoted by $D_{n} \subset S^{1}$. Then $|e - \xi| \gtrsim (n!)^{-d}$ and $|e - (0,1)| \leq 2(n!)^{-3}$ for distinct directions $e,\xi \in D_{n}$. Most importantly,
\begin{displaymath} \card \rho_{\xi}(S_{n}) \leq 3(n!)^{1 + d}, \qquad \xi \in D_{n}, \: n \geq 3. \end{displaymath}
\end{itemize}
\end{lemma}

\begin{proof} We will prove both the claims in (i) for $N(\rho_{e}(U_{n}),\delta)$ instead of $N(\rho_{e}(B_{n}),\delta)$: this is fine, since $N(\rho_{e}(B_{n}),\delta) \leq N(\rho_{e}(U_{n}),\delta)$ for any $e \in S^{1}$ and $\delta > 0$. Fix $n \in \N$ and let $(n!)^{-2} \leq \delta \leq 1$. We pursue an estimate for $\log N(\rho_{e}(U_{n}),\delta)/-\log \delta$. Let $m = m_{\delta} \in \N$ be the greatest number such that $[(m - 1)!]^{-2} > \delta$. Then $m \leq n$. Denote by $S_{U_{m}}$ the skeleton of $U_{m}$: thus, $S_{U_{m}}$ is the collection of the $(m!)^{2}$ midpoints of the squares, which form $U_{m}$. The first estimate in Lemma \ref{projCover} clearly also holds for the $\star$-products of square unions, so we have
\begin{displaymath} \card \rho_{e}(S_{U_{m}}) \leq \prod_{j = 1}^{m} \card \rho_{e}(S_{Q_{j}}), \end{displaymath}
where $S_{Q_{j}}$ is the skeleton of $Q_{j}$. Now, recalling Lemma \ref{grid} and observing that $S_{Q_{j}}$ is a dilated copy of $\{1,\ldots,j\} \times \{1,\ldots,j\} \subset \R^{2}$, we have
\begin{equation}\label{form13} \card \rho_{e}(S_{U_{m}}) \leq \prod_{j = 1}^{m} [(1 + p)(1 + q)j] = [(1 + p)(1 + q)]^{m} \cdot m! \end{equation}
for the rational direction $e = c(1,pq^{-1}) \in S^{1}$. Write $c_{p,q} := (1 + p)(1 + q)$. The side-lengths of the squares forming $U_{m}$ equal $(m!)^{-2}$, so the projection $\rho_{e}(U_{m})$ consists of intervals of length no more than $2(m!)^{-2} \leq 2\delta$, whose midpoints lie in the set $\rho_{e}(S_{U_{m}})$. These intervals can be covered by $\leq 4c_{p,q}^{m} \cdot m!$ intervals of length $\delta$, which combined with the well-known fact $\log m! \asymp m \log m$ yields
\begin{align*} \frac{\log N(\rho_{e}(U_{n}),\delta)}{-\log \delta} & \leq \frac{\log N(\rho_{e}(U_{m}),\delta)}{\log ([(m - 1)!]^{2})} \leq \frac{\log (4c_{p,q}^{m}m!)}{\log [(m - 1)!]^{2}}\\
& \lesssim \frac{m\log c_{p,q} + m \log m}{2(m - 1) \log (m - 1)} =: E(m). \end{align*} 
Now, note that $E(m) \to 1/2$ as $m \to \infty$. But $m = m_{\delta} \to \infty$ as $\delta \to 0$, whence the first inequality in (i) follows.

The second inequality in (i) is an immediate consequence of the first. Given $t > (d + 1)/(d + 2)$, apply the first inequality with
\begin{displaymath} s = s(t) := \frac{(t - 1)(2 + d) + 2}{2} = \frac{(2 + d)t}{2} - \frac{d}{2} > \frac{1 + d}{2} - \frac{d}{2} = \frac{1}{2}, \end{displaymath}
to conclude that
\begin{displaymath} N(\rho_{e}(U_{n}),\delta) \lesssim_{e,t} \delta^{-s(t)}, \qquad (n!)^{-2} \leq \delta \leq 1. \end{displaymath}
If $(n!)^{-2 - d} \leq \delta \leq (n!)^{-2}$, we first apply the previous inequality with interval length $(n!)^{-2}$ to find $\lesssim_{e,t} (n!)^{2s(t)}$ intervals of length $(n!)^{-2}$, which cover $\rho_{e}(U_{n})$. Then we split these intervals into $\leq 2(n!)^{-2}/\delta$ intervals of length $\delta$ to obtain a covering of $\rho_{e}(U_{n})$ with $\delta$-intervals of cardinality $\lesssim_{e,t} (n!)^{2s(t) - 2}/\delta$. All this yields
\begin{align*} N(\rho_{e}(U_{n}),\delta)\delta^{t} \lesssim_{e,t} (n!)^{2s(t) - 2}\delta^{t - 1} \leq (n!)^{2s(t) - 2}(n!)^{(1 - t)(2 + d)} = 1 \end{align*}
by the choice of $s(t) > 1/2$.

The inequality in (ii) follows from the estimate \eqref{form13}, which shows that
\begin{displaymath} \lim_{n \to \infty} \frac{\log \card \rho_{e}(S_{U_{n}})}{\log n!} \leq 1. \end{displaymath}  
for any fixed rational direction $e = c(1,pq^{-1}) \in S^{1}$. In particular, since $(t - 1)(2 + d) + 2 > 1$, we have
\begin{displaymath} \card \rho_{e}(S_{U_{n}}) \leq (n!)^{(t - 1)(2 + d) + 2} \end{displaymath}
for sufficiently large $n \in \N$. Then, according to the first estimate in Lemma \ref{projCover}, it follows that
\begin{align*} \card \rho_{e}(S_{n}) := \card \rho_{e}(S_{B_{n}}) & \leq [\card \rho_{e}(S_{U_{n}})] \cdot [\card S_{L_{n}}]\\
& \leq (n!)^{(t - 1)(2 + d) + 2} \cdot (n!)^{d} = (n!)^{t(2 + d)} \end{align*}
for sufficiently large $n \in \N$. 

Everything about (iii) is an immediate consequence of the definition of the directions $\xi \in D_{n}$ except for the estimate $\card \rho_{\xi}(S_{n}) \leq 3(n!)^{1 + d}$. To prove this, we need
\begin{lemma}\label{xCoord} Let $(x,y) \in S_{n}$. Then $x = (r + 1/2)(n!)^{-2}$ for some $r \in \N$.
\end{lemma}
\begin{proof} Easy induction. \end{proof}

The estimate in (iii) will follow from
\begin{claim}\label{lineIntersect} Let $n \geq 3$, and let $L$ be a line with negative slope $k(n!)^{-d}$ for some $k \in \{1,\ldots,n^{d - 2}\}$. Then either $L$ has empty intersection with $S_{n}$, or $L$ meets
\begin{displaymath} S^{+}_{n} := S_{n} \cup [S_{n} + (0,(n!)^{-2})] \cup [S_{n} - (0,(n!)^{-2})] \end{displaymath}
in a set of $n!$ points.
\end{claim}

\begin{figure}[h!]
\begin{center}
\includegraphics{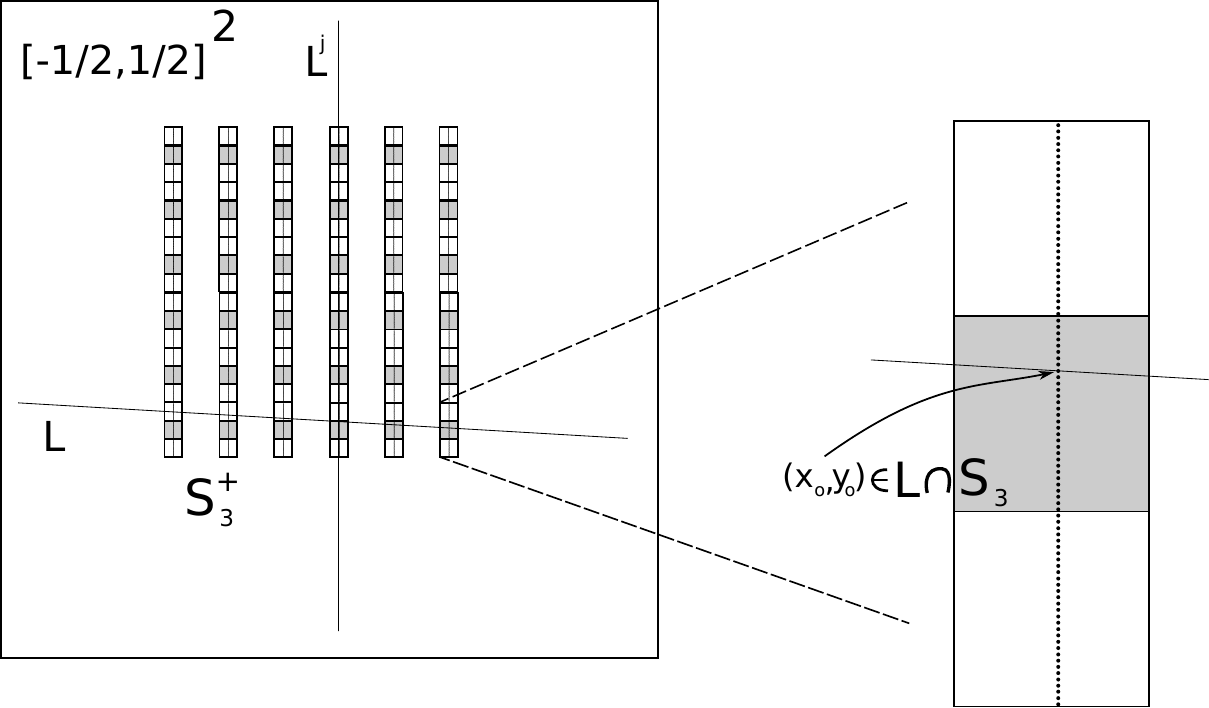}
\caption{An illustration of the set $S^{+}_{3}$ and a line $L$ with negative slope $k(n!)^{-d}$ meeting $S_{3}$. The grey squares are in $U_{3}$, but the white squares are shown only for artistic reasons: the set $S_{3}^{+}$ consists not of the squares, but the small dots inside them.}\label{S+}
\end{center}
\end{figure}

See Figure \ref{S+} for a picture of the set $S^{+}_{3}$. Let us finish the proof of (iii), assuming this claim. Note that the set $S_{n}^{+}$, $n \geq 3$, consists of $3(n!)^{2 + d}$ points, since the three sets in the definition of $S_{n}$ are disjoint for $n \geq 3$ and contain $(n!)^{2 + d}$ points each. Now suppose that $\xi \in D_{n}$ and $t \in \rho_{\xi}(S_{n})$. This means that $L := \rho_{\xi}^{-1}\{t\}$, a line with negative slope $k(n!)^{-d}$, intersects $S_{n}$. Then, according to the claim, $\card [L \cap S_{n}^{+}] = n!$. For distinct $t,t' \in \rho_{\xi}(S_{n})$, the sets $L \cap S_{n}^{+}$ are disjoint and contained in $S_{n}^{+}$. Thus,
\begin{displaymath} 3(n!)^{2 + d} = \card S_{n}^{+} \geq \card \rho_{\xi}(S_{n}) \cdot n!, \end{displaymath}
which gives the required estimate. 

Now we just need to verify Claim \ref{lineIntersect}. Let $L$ be a line with negative slope $k(n!)^{-d}$, $k \in \{1,\ldots,(n!)^{d - 2}\}$, such that $L \cap S_{n} \neq \emptyset$. Observe that all the points in $S_{n}$ lie on $n!$ vertical lines $L^{1},\ldots,L^{n!}$, and, according to Lemma \ref{xCoord}, the difference between the $x$-coordinates of any pair of these lines has the form $r(n!)^{-2}$ for some number $r \in \Z$: this difference has absolute value at most one, so we have  $|r| \leq (n!)^{2}$. Since $L$ itself is not vertical, $L$ intersects every one of the lines $L^{j}$: what we need to prove is that the point in $L \cap L^{j}$ is contained in $S_{n}^{+}$ for $1 \leq j \leq n!$. Here comes the key feature of the set $S_{n}^{+}$: if $(x_{o},y_{o}) \in S_{n}$, then
\begin{equation}\label{form14} \{y : (x,y) \in S_{n}^{+} \cap L^{j}\} \supset \{y_{o} + s(n!)^{-2 - d} : s \in \Z \text{ and } |s| \leq (n!)^{d}\} \end{equation}
for any $1 \leq j \leq n!$. In other words, for any $j \in \{1,\ldots,n!\}$, the $y$-coordinates of the set $S_{n}^{+} \cap L^{j}$ contain all the rationals of the form $y_{o} + s(n!)^{-2 - d}$, $|s| \leq (n!)^{d}$. This property follows immediately from the definitions of $S_{n}$ and $S_{n}^{+}$, and, in particular, the fact that the $y$-coordinates of the $(n!)^{d}$ points of $S_{n}$ inside any given square in $U_{n}$ are are placed at intervals $(n!)^{-2 - d}$ (see the 'magnification' on the right half of Figure \ref{S+}). 

To prove Claim \ref{lineIntersect}, fix $(x_{o},y_{o}) \in L \cap S_{n}$: such a point exists by assumption. Let $(x,y)$ be the intersection of $L$ with any line $L^{j}$, $1 \leq j \leq n!$. Then we have $x = x_{o} + r(n!)^{-2}$ for some $r \in \Z$ with $|r| \leq (n!)^{2}$. Hence, by definition of $L$,
\begin{displaymath} y = y_{o} - k(n!)^{-d}r(n!)^{-2} = y_{o} - kr(n!)^{-2 - d}. \end{displaymath}
Now it suffices to note that $kr \in \Z$ and $|kr| \leq (n!)^{2}(n!)^{d - 3} = (n!)^{d - 1} \leq (n!)^{d}$. According to \eqref{form14}, this shows that $(x,y) \in S_{n}^{+} \cap L^{j}$, and Claim \ref{lineIntersect} is proven.
\end{proof}

Now, as we start to construct the sets $K$ and $E$ of Theorem \ref{bigEx}, we may forget (almost) all about the sets $B_{n}$, and only keep in mind the properties listed in the previous lemma. Fix $\sigma \in (3/4,1)$ as in the statement of Theorem \ref{bigEx}, then choose $d \in \N$ with $d \geq 3/(1 - \sigma) > 3$. Also, pick a number $\tau = \tau(\sigma) \in ((d + 1)/(d + 2),1)$. We are now prepared to construct a compact set $K \subset B(0,1/2)$ and an exceptional set $E \subset S^{1}$ such that $\dim_{\p} E \geq \sigma$, and $\dim_{\p} K_{e} \leq \tau(\sigma)$ for every direction $e \in E$. In fact, we will even prove that $\overline{\dim}_{\B} K_{e} \leq \tau(\sigma)$ for $e \in E$, but this 'strengthening' is nothing but cosmetic according to Lemma \ref{packingIsBox}. The constructions of $K$ and $E$ proceed by induction. In our situation, however, it seems awkward to use linear induction along the natural numbers: a more flexible index set is a \emph{tree}. This is a graph $T$ with with a \emph{root} vertex $r \in T$ such that every vertex $v \in T$ has $(n_{v}!)^{d - 3}$ \emph{children} for some $n_{v} \in \N$.\footnote{The number $n_{v}$ of children will be chosen recursively, so it is not exactly well-defined to speak of the tree $T$ at this point: the infinite tree $T$ will be the end result of our induction.} Every vertex $v \in T \setminus \{r\}$ also has a unique \emph{parent} $p(v) \in T$ in the tree. The \emph{height} of a vertex $v \in T$, denoted by $h(v) \in \N$, is the distance of $v$ to the root vertex in the tree metric: thus $h(r) = 0$, and $h(v) = h(p(v)) + 1$ for $v \in T \setminus \{r\}$. To each vertex $v \in T$ we will, by a recursive procedure, associate the following items:
\begin{itemize} 
\item[(i)] a rational direction $e_{v} \in S^{1}$ and a number $c_{v} \in [1,2)$,
\item[(ii)] a compact set $K_{v} \subset B(0,1/2)$, which is the union of a collection of $k_{v} \geq h(v)$ closed balls with disjoint interiors and common diameter $\delta_{v} = k_{v}^{-1}$,
\item[(iii)] a closed arc $I_{v} \subset S^{1}$ of length $\calH^{1}(I_{v}) = \delta_{v}$, the midpoint of which is $e_{v}$.
\end{itemize}
Here are the desired properties of these parameters:
\begin{itemize}
\item[(iv)] The arcs $I_{v}$ are either nested or disjoint. If $v,w \in T$, then $I_{v} \subset I_{w}$, if and only if $v$ is a direct descendant of $w$. 
\item[(v)] All the sets $K_{v}$, $v \in T$, are nested (but we might well have $K_{v} = K_{w}$ for two distinct vertices $v,w \in T$). In particular, if $V \subset T$ is a finite collection of vertices, there exists $b \in V$ such that $K_{b} \subset K_{w}$ for all $w \in V$.
\item[(vi)] If $v \in T$, then the $(n_{v}!)^{d - 3}$ points $e_{w}$ corresponding to the children of $v$ lie in $I_{v}$ and are at distance $\gtrsim (n_{v}!)^{- d}$ from each other. 
\item[(vii)] If $e \in I_{v}$, then 
\begin{displaymath} N(K_{e},c_{v}\delta) \lesssim \delta^{-\tau}, \qquad \delta_{v} \leq \delta \leq 1. \end{displaymath}
\end{itemize}

Once we manage to get so far, we will set
\begin{displaymath} K := \bigcap_{v \in T} K_{v} \subset B(0,1) \quad \text{and} \quad E := \bigcap_{n = 0}^{\infty} \bigcup_{h(v) = n} I_{v} \subset S^{1}. \end{displaymath}
Let us quickly see how it follows from (vi) and (vii) that $\dim_{\p} K_{e} \leq \tau$ for $e \in E$ and $\dim_{\p} E \geq \sigma$. If $e \in E$, then $e \in I_{v}$ for infinitely many vertices $v \in T$. Since $\delta_{v} \to 0$ as $h(v) \to \infty$, we see immediately from (vii) that $N(K_{e},\delta)\delta^{\tau} \lesssim 1$ for all $\delta \in (0,1]$. To see that $\dim_{\p} E \geq \sigma$, one uses (vi), the information $d \geq 3/(1 - \sigma)$, and the same argument that proved in Construction \ref{setE} that the exceptional set there had packing dimension one. 

Let us initiate the construction. At first, our tree contains only one vertex, the root $r$. We start by defining $e_{r}$, $c_{r}$ and $K_{r}$: note that, by (iii), the arc $I_{r} \subset S^{1}$ is then uniquely determined by these parameters. We set $e_{r} = (0,1)$ and $c_{r} = 1$. The set $K_{r}$ is defined as the union of \textbf{the} $k_{r} \in \N$ closed balls $B \subset B(0,1/2)$ with disjoint interiors and diameter $\delta_{r} = k_{r}^{-1}$, whose centers lie on the line segment $[-1/2,1/2]$. How large should we take $k_{r}$? Lemma \ref{heart}(i) applied with $e = e_{r} = (0,1)$ implies that there exists a constant $c_{\tau} > 0$ such that
\begin{displaymath} N(\rho_{e_{r}}(B_{n}),\delta) \leq c_{\tau}\delta^{-\tau}, \qquad (n!)^{-2 - d} \leq \delta \leq 1. \end{displaymath}
Note that $S_{K_{r}} \subset \R$, so $\rho_{e_{r}}(S_{K_{r}}) = \{0\}$. Using Lemma \ref{projCover}, this implies that
\begin{displaymath} N(\rho_{e_{r}}(K_{r} \star B_{n}),\delta) = 1 \leq \delta^{-\tau}, \qquad \delta_{r} \leq \delta \leq 1, \end{displaymath}
and
\begin{displaymath} N(\rho_{e_{r}}(K_{r} \star B_{n}),\delta) \leq N\left(\rho_{e_{r}}(B_{n}),\frac{\delta}{\delta_{r}}\right) \leq [c_{\tau}\delta_{r}^{\tau}] \cdot \delta^{-\tau}, \quad \delta_{r}(n!)^{-2 - d} \leq \delta < \delta_{r}. \end{displaymath}
Now, we choose $k_{r} \in \N$ so large that $c_{\tau}\delta_{r}^{\tau} = c_{\tau}/k_{r}^{\tau} \leq 1$. Then the previous inequalities combined show that
\begin{equation}\label{form4} N(\rho_{e_{r}}(K_{r} \star B_{n}),\delta) \leq \delta^{-\tau}, \qquad \delta_{r}(n!)^{-2 - d} \leq \delta \leq 1, \end{equation}
for any $n \in \N$. 

Now $e_{r}$, $I_{r}$ and $K_{r}$ have been defined. Before we proceed, let us introduce one last piece of notation. If $e \in S^{1}$, let $R_{e} \colon \R^{2} \to \R^{2}$ be the rotation, which takes $(0,1)$ to $e$. If $n \in \N$, we write $B_{n,e} := R_{e}(B_{n})$. Now we will formulate an induction hypothesis:
\begin{itemize}
\item[(IND)] Suppose that we have already constructed a finite tree $T_{0}$ and associated to each vertex $v \in T_{0}$ the parameters $e_{v}$, $K_{v}$ and $I_{v}$ so that properties (i)--(v) hold. Moreover, if $v \in T_{0}$ is not a leaf vertex,\footnote{That is, if $v$ has children in $T_{0}$} then suppose that the number of children is $(n_{v}!)^{d - 3}$ for some $n_{v} \in \N$, and (vi) holds for $v$. According to (v), there exists $b \in T_{0}$ such that $K_{b} \subset K_{v}$ for all $v \in T_{0}$. We \textbf{assume} that
\begin{displaymath} N(\rho_{\xi}(K_{b} \star B_{n,e}),c_{b}\delta) \leq \delta^{-\tau}, \qquad \delta_{b}(n!)^{-2 - d} \leq \delta \leq 1 \end{displaymath}
for every pair of directions $e,\xi \in \{e_{v} : v \in T_{0}\}$ and for every $n \in \N \cup \{0\}$.
\end{itemize}

The content of \eqref{form4} is precisely that the the parameters associated with the root vertex $r \in T$ satisfy (IND) (and (IND) is the reason why we could not initiate the induction in any simpler manner). Pick any leaf vertex $v \in T_{0}$. Next, we will define $n_{v}$, the number of children of $v$ in $T$, and determine the values of $K_{w}$, $e_{w}$, $c_{w}$ and $I_{w}$ for all the children $w$. All of this has to be done so that (IND) remains valid for the augmented tree $T_{0} \cup \{w :  p(w) = v\}$. Already now, we mention that for every child $w$ of $v$, the set $K_{w}$ and the number $c_{w}$ will be the same, but the directions $e_{w}$ will be distinct. 

Let $n \in \N$, and consider the directions $D_{n}$ defined in Lemma \ref{heart}(iv). If $\xi \in D_{n}$, recall that $|\xi - (0,1)| \leq 2(n!)^{-3}$. Thus, the rotated directions $R_{e_{v}}(\xi)$, $\xi \in D_{n}$, satisfy $|R_{e_{v}}(\xi) - e_{v}| \leq 2(n!)^{-3}$. This shows that we may pick $n = n_{v}$ so large $R_{e_{v}}(D_{n}) \subset \operatorname{int} I_{v}$. The rational directions $e_{w}$ corresponding to the children of $v$ in $T$ are now defined to be the directions in $R_{e_{v}}(D_{n_{v}})$:
\begin{displaymath} \{e_{w} : p(w) = v\} = R_{e_{v}}(D_{n_{v}}). \end{displaymath}
Note that the distance between distinct $e_{w}$ is $\gtrsim (n!)^{-d}$ according to Lemma \ref{heart}(iii): thus (vi) holds for $v$.

As we hinted much earlier, in \eqref{form15} to be precise, the set $K_{w}$ (for any child $w$ of $v$) will have the form $K_{w} = (K_{b} \star B_{n_{v},e_{v}})^{(m_{v})}$ for some large $m_{v},n_{v} \in \N$. One criterion for the size of $n_{v}$ was already given, but there are more. Denote by $S^{n}$, $n \in \N$, the skeleton of $K_{b} \star B_{n,e_{v}}$, and write $S_{n,e}$, $e \in S^{1}$, $n \in \N$, for the skeleton of $B_{n,e}$: thus $S_{n,e} = R_{e}(S_{n})$, where $S_{n}$ is -- as before -- the skeleton of $B_{n}$. Then choose some $t \in ((d + 1)/(d + 2),\tau)$. According to Lemmas \ref{projCover} and \ref{heart}(ii), (iii), we may choose $n_{v} \in \N$ so large that
\begin{align} \card \rho_{\xi}(S^{n_{v}}) & \leq [\card \rho_{\xi}(S_{K_{b}})] \cdot [\card \rho_{\xi}(S_{n_{v},e_{v}})]\notag\\
&\label{form5} \leq [\card S_{K_{b}}] \cdot [\card \rho_{R_{e_{v}}^{-1}(\xi)}(S_{n_{v}})] \leq (n_{v}!)^{t(2 + d)} \end{align}
for all directions $\xi \in \{e_{w} : w \in T_{0}\} \cup R_{e_{v}}(D_{n_{v}})$: the reason is simply that Lemma \ref{heart}(ii) can be applied to the \textbf{finite} collection $R_{e_{v}}^{-1}(\{e_{w} : w \in T_{0}\})$ of rational directions, and the vectors $\xi \in R_{e_{v}}(D_{n_{v}})$ are handled using the bound in Lemma \ref{heart}(iii). The size of the constant $\card S_{K_{b}}$ has no bearing on the result: we can first apply Lemma \ref{heart}(ii) and (iii) with some $t'$ slightly smaller than $t$ to obtain $\card \rho_{R_{e_{v}}^{-1}(\xi)}(S_{n_{v}}) \leq (n_{v}!)^{t'(2 + d)}$ for all vectors $\xi$ as above, and then note that $[\card S_{K_{b}}] \cdot [\card \rho_{R_{e_{v}}^{-1}(\xi)}(S_{n_{v}})] \leq (n_{v}!)^{t(2 + d)}$ for $n_{v}$ large enough, of course depending on $\card S_{K_{b}}$. 

There will be three more conditions on the size of $n_{v}$. Let
\begin{displaymath} s(\tau) := \frac{(\tau - 1)(2 + d) + 2}{2} = \frac{(2 + d)\tau}{2} - \frac{d}{2} > \frac{1 + d}{2} - \frac{d}{2} = \frac{1}{2}, \end{displaymath}
and choose $1/2 < s < s(\tau)$. According to Lemma \ref{heart}(i), there exists a constant $\delta_{s} > 0$ such that
\begin{displaymath} N(\rho_{e_{v}}(B_{n,e_{v}}),\delta) = N(\rho_{(0,1)}(B_{n}),\delta) \leq \delta^{-s}, \qquad (n!)^{-2} \leq \delta \leq \delta_{s}. \end{displaymath}
This combined with Lemma \ref{projCover} shows that
\begin{align*} N(\rho_{e_{v}}(K_{b} \star B_{n,e_{v}}),\delta) & \leq [\card \rho_{e_{v}}(S_{K_{b}})] \cdot N\left(\rho_{e_{v}}(B_{n,e_{v}}),\frac{\delta}{\delta_{b}}\right)\\
& \leq [\card S_{K_{b}} \cdot \delta_{b}^{s}] \cdot \delta^{-s}, \qquad \delta_{b}(n!)^{-2} \leq \delta \leq \delta_{b}\delta_{s}. \end{align*} 
Now we have to, first, choose $n = n_{v}$ so large that $(n_{v}!)^{-2} \leq \delta_{s}$ and, second, so large that $[\card S_{K_{b}}] \cdot (n_{v}!)^{2s} \leq (n_{v}!)^{2s(\tau)}/2$. Then the previous inequality applied with $\delta = \delta_{b}(n_{v}!)^{-2}$ gives
\begin{equation}\label{form7} N(\rho_{e_{v}}(K_{b} \star B_{n_{v},e_{v}}),\delta_{b}(n_{v}!)^{-2}) \leq [\card S_{K_{b}} \cdot \delta_{b}^{s}] \cdot (\delta_{b}(n_{v}!)^{-2})^{-s} \leq (n_{v}!)^{2s(\tau)}/2. \end{equation} 
The final condition on $n_{v}$ is this: $n_{v}$ must be chosen so large that
\begin{displaymath} c_{w} := c_{b}\left(1 + \frac{2(n_{v}!)^{-3}}{c_{b}\delta_{b}(n_{v}!)^{-2}}\right) < 2. \end{displaymath}
As we remarked earlier, this definition of $c_{w}$ is common for all the children $w$ of $v$. Now we are ready to prove that
\begin{equation}\label{form6} N(\rho_{\xi}(K_{b} \star B_{n_{v},e_{v}}),c_{w} \delta) \leq \delta^{-\tau}, \qquad \delta_{b}(n_{v}!)^{-2 - d} \leq \delta \leq 1 \end{equation}
for all $\xi \in \{e_{w} : w \in T_{0}\} \cup R_{e_{v}}(D_{n_{v}})$. If $\xi = e_{w}$ for some $w \in T_{0}$, then \eqref{form6} holds by (IND), since $c_{w} \geq c_{b}$. So, let $\xi = R_{e_{v}}(d)$ for some $d \in D_{n_{v}}$. As noted before, $\xi$ satisfies the estimate $|\xi - e_{v}| \leq 2(n_{v}!)^{-3}$. It follows from this and the definition of $c_{w}$ that if $\delta_{b}(n_{v}!)^{-2} \leq \delta \leq 1$, and $\rho_{e_{v}}(K_{b} \star B_{n_{v},e_{v}})$ can be covered with, say, $k$ intervals of length $c_{b}\delta$, then $\rho_{\xi}(K_{b} \star B_{n_{v},e_{v}})$ can be covered by the $k$ intervals with the same midpoints but the slightly larger length $c_{w}\delta$.\footnote{Any definition of $c_{w} \in (1,2)$ such that this requirement is satisfied would be ok, so there is no further magic behind the complicated looking definition.} In other words,
\begin{displaymath} N(\rho_{\xi}(K_{b} \star B_{n_{v},e_{v}}),c_{w} \delta) \leq N(\rho_{e_{v}}(K_{b} \star B_{n_{v},e_{v}}),c_{b}\delta) \leq \delta^{-\tau}, \qquad \delta_{b}(n_{v}!)^{-2} \leq \delta \leq 1. \end{displaymath}
But this is not quite \eqref{form6} yet. Next, let $\delta_{b}(n_{v}!)^{-2 - d} \leq \delta < \delta_{b}(n_{v}!)^{-2}$. According to \eqref{form7}, the set $\rho_{e_{v}}(K_{b} \star B_{n_{v},e_{v}})$ can be covered with $(n_{v}!)^{2s(\tau)}/2$ intervals of length $\delta_{b}(n_{v}!)^{-2}$: note that this estimate is slightly better than the previous bound applied with $\delta = \delta_{b}(n_{v}!)^{-2}$. Once more exploiting the fact $|\xi - e_{v}| \leq 2(n_{v}!)^{-3}$ and the definition of $c_{w}$, the same intervals amplified by a factor of $c_{w}$ suffice to cover $\rho_{\xi}(K_{b} \star B_{n_{v},e_{v}})$. A covering of $\rho_{\xi}(K_{b} \star B_{n_{v},e_{v}})$ with $c_{w}\delta$-intervals is then simply obtained by splitting all the intervals of length $c_{w}\delta_{b}(n_{v}!)^{-2}$ into $2\delta_{b}(n_{v}!)^{-2}/\delta$ intervals of length $c_{w}\delta$. The total number of $c_{w}\delta$-intervals required to cover $\rho_{\xi}(K_{b} \star B_{n_{v},e_{v}})$ is hence bounded above by $(n_{v}!)^{2s(\tau)}\delta_{b}(n_{v}!)^{-2}/\delta$, which gives
\begin{align*} N(\rho_{\xi}(K_{b} \star B_{n_{v},e_{v}}),c_{w}\delta)\delta^{\tau} & \leq (n_{v}!)^{2s(\tau)}\delta_{b}(n_{v}!)^{-2}\delta^{\tau - 1}\\
& \leq (n_{v}!)^{2s(\tau)}\delta_{b}(n_{v}!)^{-2}[\delta_{b}(n_{v}!)^{-2 - d}]^{\tau - 1}\\
& = \delta_{b}^{\tau}(n_{v}!)^{2s(\tau) - (\tau - 1)(2 + d) - 2} = \delta_{b}^{\tau} \leq 1. \end{align*}
This proves \eqref{form6} and finishes the definition of $n_{v}$. Now that the number of children of $v$ has been permanently determined, it is certainly well-defined to write  $T_{+} := T_{0} \cup \{w : p(w) = v\}$.

It remains to fix $m_{v} \in \N$. Recall that $S^{n_{v}}$ was the skeleton of $K_{b} \star B_{n_{v},e_{v}}$. If $(\delta_{b}(n_{v}!)^{-2 - d})^{2} \leq \delta \leq \delta_{b}(n_{v}!)^{-2 - d}$, Lemma \ref{projCover} combined with the estimates \eqref{form5} and \eqref{form6} yields
\begin{align*} N(\rho_{\xi}[(K_{b} \star B_{n_{v},e_{v}})^{(2)}],c_{w}\delta) & \leq [\card \rho_{\xi}(S^{n_{v}})] \cdot N\left(\rho_{\xi}(K_{b} \star B_{n_{v},e_{v}}),c_{w}\left[\frac{\delta}{\delta_{b}(n_{v}!)^{-2 - d}}\right]\right)\\
& \leq (n_{v}!)^{t(2 + d)} \cdot \left(\frac{\delta}{\delta_{b}(n_{v}!)^{-2 - d}}\right)^{-\tau} \leq \delta^{-\tau} \end{align*}
for all directions $\xi \in \{e_{w} : w \in T_{+}\}$, and the same inequality for $\delta_{b}(n_{v}!)^{-2 - d} \leq \delta \leq 1$ follows immediately from \eqref{form6}. This reasoning can be iterated to show that
\begin{equation}\label{form8} N(\rho_{\xi}[(K_{b} \star B_{n_{v},e_{v}})^{(m)}],c_{w}\delta) \leq \delta^{-\tau}, \qquad (\delta_{b}(n_{v}!)^{- d - 2})^{m} \leq \delta \leq 1, \end{equation}
for any $m \in \N$ and for all $\xi \in \{e_{w} : w \in T_{+}\}$. We are finally close to proving (IND) for the set $K_{w} := (K_{b} \star B_{n_{v},e_{v}})^{(m)}$ for some sufficiently large $m \in \N$. We remind the reader that the set $K_{w}$ is the same for all the children $w$ of $v$; also, after $K_{w} \subset K_{b}$ is constructed, it will be clearly be the smallest set (in terms of inclusion) in the augmented tree $T_{+}$. Thus, according to (IND), we should be able to prove that
\begin{equation}\label{form9} N(\rho_{\xi}(K_{w} \star B_{p,e}),c_{w}\delta) \leq \delta^{-\tau}, \qquad \delta_{w}(p!)^{-2 - d} \leq \delta \leq 1, \end{equation}
for any $p \in \N$ and for any pair of directions $e,\xi \in \{e_{w} : w \in T_{+}\}$. Here $\delta_{w} = (\delta_{b}(n_{v}!)^{-d - 2})^{m}$ is the diameter of the balls in $K_{w}$. Fix $p \in \N$ and $e,\xi \in \{e_{w} : w \in T_{+}\}$. There are only finitely many such pairs, and all the directions are rational, so it follows from the latter estimate in Lemma \ref{heart}(i) that
\begin{equation}\label{form10} N(\rho_{\xi}(B_{p,e}),\delta) \leq C_{T_{+}} \cdot \delta^{-\tau}, \qquad (p!)^{-2 - d} \leq \delta \leq 1 \end{equation}
for some constant $C_{T_{+}} > 0$ depending only on these finitely many rational configurations. Now, if we denote by $S^{n_{v},m}$ the skeleton of the set $K_{w}$, inequality \eqref{form5} and the first estimate in Lemma \ref{projCover} combine to produce the bound
\begin{equation}\label{form11} \card \rho_{\xi}(S^{n_{v},m}) \leq (n_{v}!)^{mt(2 + d)}, \qquad m \in \N. \end{equation}
Fix $\delta_{w}(p!)^{-2 - d} \leq \delta \leq 1$. If $\delta \geq \delta_{w}$, then \eqref{form9} follows immediately from \eqref{form8}. In case $\delta < \delta_{w}$ we resort to Lemma \ref{projCover} once more. This combined with \eqref{form10} and \eqref{form11} yields
\begin{align*} N(\rho_{\xi}(K_{w} \star B_{p,e}),c_{w}\delta) & \leq [\card \rho_{\xi}(S^{n_{v},m})] \cdot N\left(\rho_{\xi}(B_{p,e}),\frac{\delta}{\delta_{w}}\right)\\
& \leq (n_{v}!)^{mt(2 + d)} \cdot C_{T_{+}} \cdot \left(\frac{\delta}{\delta_{w}}\right)^{-\tau}\\
& = C_{T_{+}} \cdot (n_{v}!)^{m(2 + d)(t - \tau)} \cdot \delta^{-\tau} \end{align*}
Now, the only condition we place on $m = m_{v}$ is that $C_{T_{+}} \cdot (n_{v}!)^{m(2 + d)(t - \tau)} \leq 1$. This can be achieved, since $t < \tau$. With this choice of $m_{v}$, the set $K_{w}$ satisfies \eqref{form9} and, consequently, (IND). To finish the entire construction, there remains the minor point that the intervals $I_{w}$, $w \in T_{+}$, have to be disjoint. Recall that, for the children $w$ of $v$, the directions $e_{w}$ were at least $(n_{v}!)^{-d}$ apart. This number does not depend on $m_{v}$; on the other hand $\calH^{1}(I_{w}) = \delta_{w} = (\delta_{b}(n_{v}!)^{-2 - d})^{m}$, which can be made arbitrarily small by increasing $m_{v}$ only. 

Right after formulating the properties (i)--(vii), we demonstrated that the proof of Theorem \ref{bigEx} would be finished (except for the part about $\calH^{1}(K) > 0$) given these properties for $K$ and $E$. Now (IND) states directly that properties (i)--(vi) are in force: what about (vii)? Let $v \in T$, and let $e \in I_{v}$. During the construction of the tree $T$, there comes a point where $K_{v}$ is the smallest set in the finite subtree constructed so far: in the terms of (IND), we have $v = b$ with respect to some subtree $T_{0} \subset T$. Then (IND) applied with $n = 0$ (then $B_{n,e} = B(0,1/2)$) shows that
\begin{equation}\label{form12} N(K_{e_{v}},c_{v}\delta) \leq N(\rho_{e_{v}}(K_{v}),c_{v}\delta) \leq \delta^{-\tau}, \qquad \delta_{v} \leq \delta \leq 1. \end{equation}
Since $e \in I_{v}$, we have $|e - e_{v}| \lesssim \delta_{v}$: this implies that the number of $\delta$-intervals required to cover $K_{e_{v}}$ is comparable to the number of $\delta$-intervals required to cover $K_{e_{v}}$ for any $\delta \geq \delta_{v}$. This observation combined with \eqref{form12} proves (vii). 

We omit the proof of $\calH^{1}(K) > 0$, since it is entirely standard. For example, in \cite[\S 4.12]{Ma} there are given conditions, which guarantee that $\calH^{s}(E) > 0$ for any $s > 0$ and any 'Cantor type' set $E$. It is easy to verify that $K$ satisfies all of these conditions with $s = 1$. The proof of Theorem \ref{bigEx} is finished.

\subsection{Proof of Proposition \ref{pointMarstrand}}

Proposition \ref{pointMarstrand} is an easy consequence of a theorem of Szemer\'edi and Trotter \cite{ST} on the number of incidences between points and lines in the plane. Let us state this estimate:
\begin{thm}[Szemer\'edi-Trotter incidence bound]\label{ST} Let $P \subset \R^{2}$ be a set of $n$ points, and let $\mathcal{L}$ be a collection of $m$ lines in $\R^{2}$. Write $I(P,\mathcal{L})$ for the set of \emph{incidences} between the points in $P$ and the lines in $L$. Formally, we define
\begin{displaymath} I(P,\mathcal{L}) := \{(p,L) : p \in P, L \in \mathcal{L} \text{ and } p \in L\}. \end{displaymath}
Then
\begin{displaymath} \card I(P,\mathcal{L}) \leq A(m^{2/3}n^{2/3} + m + n), \end{displaymath} 
where $A > 0$ is an absolute constant.
\end{thm}

Now we are armed to prove Proposition \ref{pointMarstrand}:

\begin{proof}[Proof of Proposition \ref{pointMarstrand}] Let $P \subset \R^{2}$ be a set with $n \geq 2$ points. Suppose that $S \subset S^{1}$ is a set of directions such that $\card S = k$ and $\card P_{e} \leq n^{s} < n$ for $e \in S$. Let $A > 0$ be the constant from Theorem \ref{ST}. If $n$ is so small that $n^{s - 1} > 1/(2A)$, the desired inequality follows from the trivial bound $k \leq n^{2} \lesssim_{s} n^{2s - 1}$. Thus, we may assume that $n^{s - 1} \leq 1/(2A)$. We apply the Szemer\'edi-Trotter estimate with the point set $P$ and the collection of lines
\begin{displaymath} \mathcal{L} := \{\rho_{e}^{-1}\{t\} : e \in S \text{ and } t \in P_{e}\}. \end{displaymath} 
Then every point $p \in P$ is incident with exactly $k$ lines, which yields 
\begin{displaymath} \card I(P,\mathcal{L}) = kn. \end{displaymath}
On the other hand, there are no more than $kn^{s}$ lines in $\mathcal{L}$, so that 
\begin{displaymath} kn = \card I(P,\mathcal{L}) \leq A[(kn^{s})^{2/3}n^{2/3} + kn^{s} + n] = A(2k^{2/3}n^{(2s + 2)/3} + kn^{s}). \end{displaymath}
Here we needed the assumption $s \geq 1/2$ in the form $n \leq k^{2/3}n^{(2s + 2)/3}$. Dividing by $k^{2/3}n$ and using the assumption $n^{s - 1} \leq 1/(2A)$ gives
\begin{displaymath} k^{1/3} \leq A(n^{(2s - 1)/3} + k^{1/3}n^{s - 1}) \leq An^{(2s - 1)/3} + \frac{k^{1/3}}{2}. \end{displaymath}
Move $k^{1/3}/2$ to the left hand side and raise everything to the third power to conclude the proof.
\end{proof}

\section{Open questions}

\begin{question} How sharp are the bounds in Theorem \ref{estimate}? In particular, is it true that
\begin{displaymath} \dim_{\p} \{e \in S^{1} : \dim_{\p} K_{e} < \dim K\} < 1, \end{displaymath}
if $\dim K < 1$? According to the estimate \eqref{kaufman} by Kaufman, this holds if the first $\dim_{\p}$ (or both $\dim_{\p}$'s) is replaced by $\dim$. What is the sharp behavior of the best bound for $\dim_{\p} \{e \in S^{1} : \dim_{\p} K_{e} \leq \sigma\}$, as $\sigma \searrow \dim K/2$? Should the bound tend to zero, as in Bourgain's estimate \eqref{bourgain}?
\end{question}

\begin{question}\label{Quest} What is the best estimate one can obtain for the \textbf{Hausdorff} dimension of the set $\{e \in S^{1} : \dim_{\p} K_{e} \leq \sigma\}$ for $\sigma < \dim K$? Peres, Simon and Solomyak make no comment on the sharpness of their bound \eqref{PSSEstimate}, and the Hausdorff dimension of the exceptional set in Theorem \ref{bigEx} is likely to equal zero.  Could it be that
\begin{displaymath} \dim \{e \in S^{1} : \dim_{\p} K_{e} < \dim K\} = 0, \qquad \dim K \leq 1? \end{displaymath}
\end{question}

\section{Acknowledgements}

I am thankful to Pertti Mattila and Esa Järvenpää for many useful comments. 


\end{document}